\documentclass[12pt]{amsart}
\usepackage{amsfonts}
\usepackage{amssymb}
\usepackage{color}
\usepackage{dsfont}

\numberwithin{equation}{section}
\theoremstyle{plain}
 \newtheorem{theorem}{Theorem}[section]
 \newtheorem{lemma}[theorem]{Lemma}
 \newtheorem{corollary}[theorem]{Corollary}
 \newtheorem{proposition}[theorem]{Proposition}

\theoremstyle{definition}
 \newtheorem{definition}[theorem]{Definition}
 \newtheorem{example}[theorem]{Example}
 \newtheorem{remark}[theorem]{Remark}

\newcommand{\Real}{\mathrm{Re}\,}

\newcommand{\C}{\mathbb{C}}
\newcommand{\R}{\mathbb{R}}
\newcommand{\N}{\mathbb{N}}
\newcommand{\Z}{\mathbb{Z}}

\newcommand{\bN}{\mathbb{N}}
\newcommand{\bZ}{\mathbb{Z}}
\newcommand{\bR}{\mathbb{R}}

\newcommand{\bE}{\mathbb{E}}

\allowdisplaybreaks

\begin{document}

\vspace{5mm}
\begin{center}
{\bf
{\large

Almost periodic stationary processes}}

\vspace{5mm}

David Berger and Farid Mohamed\\
\end{center}

\vspace{5mm}
We derive a necessary and sufficient condition for stochastic processes to have almost periodic finite dimensional distributions; in particular, we obtain characterizations for infinitely divisible processes to be almost periodic in terms of their characteristic triplets. Furthermore, we derive conditions when the process $(X_t)_{t\in\R}$ defined by the stochastic integral $X_t:= \int_{\R^d} f(t,s) dL(s)$ is almost periodic stationary and also when it is almost periodic in probability, where $f(t,\cdot)\in L^1(\R^d,\R)\cap L^2(\R^d,\R)$ is deterministic and $L$ is a L\'evy basis. Moreover, we discuss almost periodic Ornstein-Uhlenbeck-type processes, and obtain a central limit theorem for $m$-dependent processes with almost periodic finite dimensional distributions. 


\section{Introduction}
The study of almost periodic functions was originally motivated by considering two periodic functions with different periods. The sum of these two functions is not necessarily periodic any more but it is still almost periodic. Bohr developed 1925 in [\ref{Bohr}] the theory of almost periodic functions and also showed a Fourier series representation for almost periodic functions which differs from the periodic case as the summation is over periodic functions with different periods. Ever since the introduction by Bohr, almost periodic functions found wide applications in dynamical systems, representation theory, differential equations, and in various other branches of mathematics (see e.g. [\ref{Bezandry}], [\ref{Corduneanu}], [\ref{NGue}], [\ref{Neumann}]).\\
In the context of stochastic processes Gladyshev defined in [\ref{Gladyshev}] stochastic processes with almost periodic mean and covariance functions.\,Further concepts are for example almost periodicity in  distribution, in probability, in quadratic mean, almost sure, etc.\,In [\ref{Bedouhene}] and [\ref{Tudor1}] those different types of almost periodicity were compared and afterwards applied in order to obtain almost periodic solutions of stochastic differential equations.\,In particular, Morozan and Tudor gave in [\ref{Morozan}] a definition for almost periodic mappings with values in the space of probability measures.\,Furthermore,  in [\ref{Tudor1}] almost periodicity in distribution was discussed, which assumes the almost periodicity of the finite dimensional distributions of the process.\,Let us note that we take up the concept of almost periodic processes from [\ref{Morozan}], [\ref{Tudor1}], where we use different distance functions which metrize the topology of weak convergence on the set of probability measures.\,We call such processes almost periodic stationary.\,The goal of this paper is to find a characterization of stochastic processes $(X_t)_{t\in\R}$ to be almost periodic in distribution, especially for infinitely divisible processes via their characteristic function.\,As a concrete example we derive conditions for which the process defined by the stochastic integral (in the sense of Rajput and Rosinski)
\begin{align}
    X_t = \int\limits_{\R^d} f(t,s)dL(s)\label{stochasticIntegral1}
\end{align}
is almost periodic in distribution; moreover, we derive a sufficient condition for such processes to be almost periodic in probability (see [\ref{Bedouhene}], [\ref{Precupanu}]), where $f:\R\times\R^d\to\R$ is a deterministic function and $L$ a so-called L\'evy basis.
Further, we consider Ornstein-Uhlenbeck-type processes whose finite dimensional distributions are almost periodic and show the existence and uniqueness of such processes.\,Moreover, we derive a central limit theorem for $m$-dependent almost periodic stationary processes.\\
The paper is organized as follows. 
\begin{itemize}
    \item In Section 2 we define almost periodic stationary processes and state general results for such processes. Then, we characterize the almost periodic stationarity of stochastic processes via their characteristic function (see Theorem \ref{thechar}). 
    \item This leads us in Section 3 to our main Theorem \ref{theoremID}, which characterizes the almost periodic stationarity of infinitely divisible processes in terms of their characteristic triplets.
    \item As an application, we discuss in Section 4 the almost periodic stationarity of stochastic processes of the form (\ref{stochasticIntegral1}), see Corollary \ref{a10}. 
    \item In Section 5 we derive sufficient conditions for (\ref{stochasticIntegral1}) to be almost periodic in probability (see Theorem \ref{a1}) and consider some examples.
    \item Then, we apply our results in Section\,6 to obtain the unique almost periodic stationary solution of the Ornstein-Uhlenbeck equation (see Theorem \ref{OrnTheorem1}).
    \item Finally, in Section 7 we state the central limit theorem for $m$-dependent almost periodic stationary processes (see Theorem \ref{mcentral}).
\end{itemize}
For notation, we denote by $\mathcal{M}(\R^n)$ the collection of all finite measures on the measurable space $(\R^n,\mathcal{B}(\R^n))$ and by $\mathcal{P}(\R^n)$ the collection of all probability measures on $(\R^n,\mathcal{B}(\R^n))$, where $\mathcal{B}(\R^n)$ denotes the Borel-$\sigma$-algebra on $\R^n$. 

\section{Almost periodic stationary Processes}
The definition of an almost periodic function was first given by Bohr [\ref{Bohr}], where he introduced almost periodicity of functions which map to the real numbers. There exist several extensions by assuming that the range of the function is contained in a Banach or Hilbertspace, or even more general, in a complete separable metric space. As we are interested in functions which map to the space of measures on a complete metric space, we need the definition of almost periodic functions in such spaces, which can be found e.g. in [\ref{Morozan}], [\ref{Tudor1}].

\begin{definition}\label{DefiAPD}
Let $(M,d)$ be a complete metric space and $f:\bR\to M$ be a continuous function. We say that $f$ is \emph{almost periodic} if for every $\varepsilon>0$ there exists an $L_\varepsilon>0$ and a $\tau=\tau(a,\varepsilon)\in [a,a+L_{\varepsilon}]$ for all $a\in\bR$ such that
\begin{align}
    d(f(x+\tau),f(x))<\varepsilon\textrm{ for all }x\in\bR.\label{DefinitionAPD}
\end{align}
\end{definition}

\begin{remark}\label{Rem <3}
In [\ref{Morozan}], [\ref{Tudor1}] it is additionally assumed that $(M,d)$ is separable. It is however easy to see that for the following considerations, separability is not needed to be required. This can be seen by introducing the set $N$ as the closure of the range of the continuous function $f:\R\to M$ and restricting the metric $d$ to $N\times N$. Then $N$ is a separable complete metric space, and clearly (\ref{DefinitionAPD}) is not affected if $f$ is considered as a function to $M$ or to $N$, respectively. The same is true for the equivalent characterizations of almost periodicity by Bochner given below. 
\end{remark}

The definition of Bohr seems to imply that the metric is important for the definition of almost periodicity, but it turns out that only the generated topology is relevant.\,There exist at least two equivalent definitions, which do not include the notion of a metric, but of a topology. As a consequence, we will see that we can change the metric as long as the different metrics generate the same topology.\,The first equivalent definition of almost periodicity was given in [\ref{Bochner}], which states the following. Let $f:\bR \to M$ be a continuous function.\,Then $f$ is almost periodic if and only if the set $(f(\cdot+\tau))_{\tau\in\bR}$ is relatively compact in the space $C(\bR,M)$.\,(This property is often referred to as almost periodicity in the sense of Bochner and the equivalence with the previous definition stems from the fact that we consider functions with domain $\R$, see [\ref{Bedouhene}, p.324]). The topology on $C(\R,M)$ is the topology of uniform convergence on compact subsets of $\R$, which can be shown to be independent of the underlying metric on $M$, see e.g. [\ref{Engelking}, Theorem 8.2.6].\,As a consequence of this result, Bochner obtained another, very useful characterization of almost periodicity, the so-called \emph{Bochner's double sequence criterion}, which states the following.\,A continuous function $f:\bR\to M$ is almost periodic if and only if for every two sequences $(a_n)_{n\in\bN}$ and $(b_n)_{n\in \bN}$ there exist two subsequences $(a'_n)_{n\in\bN}$ and $(b'_n)_{n\in\bN}$ such that the limits $\lim_{n\to\infty}\lim_{m\to\infty}f(x+a'_n+b_m')$ and $\lim_{n\to\infty}f(x+a'_n+b'_n)$ exist for every $x\in\bR$ and are equal.\,For references, see [\ref{Bedouhene}], [\ref{Bochner}], [\ref{Tudor1}] and [\ref{Neumann}].

In [\ref{Morozan}] Morozan and Tudor defined stochastic processes $(X_t)_{t\in\R}$ with \emph{almost periodic finite dimensional distributions}.\,We adopt this definition, whereas we call such processes almost periodic stationary instead.\,An almost periodic stationary process is by definition a process such that the shift operator functions on the finite dimensional distributions are almost periodic in the suitable space.\\


\begin{definition}\label{def1}
Let $X=(X_t)_{t\in\bR}$ be a real-valued stochastic process and define for $n\in\N$ and $(x_1,\dotso,x_n)\in\R^n$ the function $\mu_{x_1,\dotso,x_n}:\R\to\mathcal{P}(\R^n)$, $\mu_{x_1,\dotso,x_n}(t):=\mathcal{L}(X_{x_1+t},\dotso,X_{x_n+t})$, where $\mathcal{L}(X_{x_1},\dotso,X_{x_n})$ denotes the $n$-dimensional distribution of $X$ in $(x_1,\dotso,x_n)\in\R^n$. We call $X$ \emph{almost periodic stationary} if for any $n\in\N$ and $(x_1,\dotso,x_n)\in \R^n$ the function $t\mapsto \mu_{x_1,\dotso,x_n}(t)$ is almost periodic in $t\in\R$.
\end{definition}

%




We did not mention the metric used on the space of probability measures, but we always assume that the metric is complete and induces the topology of weak convergence on $\mathcal{P}(\R^n)$.\,The following metrics satisfy this assumption and are useful for different applications.\\
Since we will need the Prokhorov metric later on also for finite measures rather than only probability measures, we give it immediately in this context. The other metrics are only used on the space of probability measures. \newpage
\begin{definition}\label{defnewmet}\, 
\begin{itemize}
    \item[a)] The \emph{Prokhorov metric} $\delta_n:\mathcal{M}(\R^n)\times \mathcal{M}(\R^n)\to [0,\infty)$ is defined as
\begin{align*}
    \delta_n(\mu,\nu) := \inf\{ \varepsilon>0 : \mu(A)\le \nu(A^{\varepsilon}) + \varepsilon \textrm{ and } \nu(A)\le \mu(A^{\varepsilon})+\varepsilon \textrm{ for all } A\in \mathcal{B}(\R^n) \},
\end{align*}
where $A^{\varepsilon} = \bigcup\limits_{x\in A} \{y\in\R^n: |y-x|<\varepsilon\}$.
    \item[b)] The \emph{bounded-Lipschitz metric} $\beta_n:\mathcal{P}(\R^n)\times \mathcal{P}(\R^n)\to [0,\infty)$ is defined as
\begin{align*}
    \beta_n(\mu,\nu) := \sup\limits_{\|f\|_{BL}\le 1} \bigg\{  \bigg| \int_{\R^n}\ f \, d(\mu-\nu) \bigg| \bigg\}, 
\end{align*}
where $$\|f\|_{BL} := \sup\limits_{x\neq y} \frac{|f(x)-f(y)|}{|x-y|} +  \sup\limits_{x\in\R^n} |f(x)|,$$ for $f:\R^n\to\R$.
    \item[c)] We define $\gamma_n:\mathcal{P}(\R^n)\times\mathcal{P}(\R^n)\to [0,\infty)$ as
\begin{align*}
    \gamma_n(\mu,\nu):= \sum\limits_{k=1}^{\infty} 2^{-k} \| \hat{\mu}-\hat{\nu} \|_{C([-k,k]^n)},
\end{align*}
where $\hat{\mu},\hat{\nu}$ are the characteristic functions of $\mu$ and $\nu$ respectively, and $C([-k,k]^n)$, $k\in\N$, is the space of continuous functions restricted to $[-k,k]^n$ endowed with the uniform norm.
\end{itemize}
\end{definition}

\begin{remark}\label{nMetric}
The metrics in a) and b) are covered extensively in [\ref{Dudley}]. The metric defined in c) is indeed a complete metric which induces the topology of weak convergence on $\mathcal{P}(\R^n)$, and it is separable.\, The proof relies on a standard argument, but as we did not find a proof of this result, we give a short proof here.
\end{remark}

\begin{proof}[Proof of Remark \ref{nMetric}]
Clearly, $\gamma_n$ is a metric on $\mathcal{P}(\R^n)$. Furthermore, we observe
\begin{align*}
    \gamma_n(\mu,\nu) &= \sum\limits_{k=1}^{\infty} 2^{-k} (1+\sqrt{n} k)\bigg\|\frac{ \hat{\mu}-\hat{\nu}}{1+\sqrt{n}k} \bigg\|_{C([-k,k]^n)}\\
    &\le  \sum\limits_{k=1}^{\infty} 2^{-k} (1+\sqrt{n}k) \cdot 2\beta_n(\mu,\nu)\\
    &= (2+4\sqrt{n})\,\beta_n(\mu,\nu).
\end{align*}
So if there is a sequence $\mu_l$, $l\in\N$, and $\mu$, all in $\mathcal{P}(\bR^n)$ such that $\mu_l$ converges weakly to $\mu$ as $l\to\infty$, we see that $\lim_{l\to\infty}\gamma_n(\mu_l,\mu)=0$.\\
Now take a sequence $(\mu_l)_{l\in\bN}\subset\mathcal{P}(\bR^n)$ such that for every $\varepsilon>0$ there exists an $m_0$ such that for all $l,m>m_0$ we have $\gamma_n(\mu_l,\mu_m)<\varepsilon$.\,We conclude that $(\hat{\mu}_l)_{l\in\bN}$ is a Cauchy sequence in every space $(C([-k,k]^n),\|\cdot\|_{C([-k,k]^n)})$.\,Hence, there exists a bounded function $\hat{\mu} \in C(\bR)$ such that $\|\hat{\mu_l}-\hat{\mu}\|_{C([-k,k])}\to 0$ as $l\to\infty$ for every $k\in\N$.\,By Levy's continuity theorem we know that $\hat{\mu}$ is the characteristic function of a probability measure $\mu$ and the sequence $\mu_l$ converges weakly to $\mu$, which implies that the limit exists in $\gamma_n$.\,This implies that the metric is indeed metricizing the weak topology on $\mathcal{P}(\bR^n)$ and it is complete. Finally, since separability of a metric depends only on its induced topology and since $\delta_n$ is separable, so is $\gamma_n$.
\end{proof}

We directly see that strictly stationary processes and periodic strictly stationary processes indeed satisfy the condition in Definition \ref{def1}.\,Furthermore, there exists a definition of almost periodically correlated processes, i.e.~a process such that its mean and covariance functions are almost periodic (see [\ref{Gladyshev}]).

\begin{definition}
A real-valued stochastic process $X=(X_t)_{t\in\R}$ with finite second moment, i.e. $\mathbb{E}X_t^2<\infty$ for all $t\in\bR$, is called \emph{almost periodically correlated} if its mean and covariance functions
\begin{align*}
    m(t) &:= \mathbb{E} X_t \quad \textrm{and}\\
    C_a(t) &:= \mathbb{E}X_t X_{t+a} - \mathbb{E}X_t \mathbb{E}X_{t+a}
\end{align*}
are almost periodic functions in $t\in\R$ for all $a\in\R$.
\end{definition}

In [\ref{Bedouhene}] and [\ref{Tudor1}] different concepts of almost periodic stochastic processes were discussed and compared. In the following we continue this line of research by showing that our definition implies almost periodic correlation if the process $X=(X_t)_{t\in\R^d}$ is $L^2$-uniformly integrable. In order to do so, we briefly discuss almost periodicity in the well-known Wasserstein metric and then mention further properties of almost periodic stationary processes.\,Let $p\in [1,\infty)$. Recall that a stochastic process $(X_t)_{t\in\R}$ is \emph{$L^p$-uniformly integrable} if to each $\varepsilon>0$ there exists a $k>0$ such that $\sup\limits_{t\in\R} \mathbb{E}|X_t|^p\mathds{1}_{|X_t|>k} < \varepsilon$.



\begin{remark}
For any two probability measures $\mu,\nu\in \mathcal{P}(\R^n)$ with finite $p$-moment ($p\in [1,\infty)$), the $p$-th \emph{Wasserstein metric} is defined as
\begin{align*}
    W_p(\mu,\nu) : = \left( \inf\limits_{\sigma \in\Gamma(\mu,\nu)} \int\limits_{\R^n\times\R^n} |x-y|^p  \,d\sigma(x,y) \right)^{1/p},
\end{align*}
where $\Gamma(\mu,\nu)$ denotes the set of all measures on $\R^n\times\R^n$ with marginals $\mu$ and $\nu$. Given a sequence $(\mu_n)_{n\in\N}$ of probability measures with finite $p$-th moment and a probability measure $\mu$ with finite $p$-th moment, we know that $W_p(\mu_n,\mu)\to 0$ for $n\to\infty$ is equivalent to $(\mu_n)_{n\in\N}$ being $L^p$-uniformly integrable and weakly convergent to $\mu$ (see [\ref{Villani}, Theorem\,7.12]).\,We observe that if a process $X=(X_t)_{t\in\R}$ is $L^p$-uniformly integrable and almost periodic stationary with respect to one of the metrics from Definition \ref{defnewmet}, then $X$ is also almost periodic stationary with respect to the metric $W_p$, which can be seen from Bochner's double sequence criterion.\,Let us now summarize some results on almost periodic stationary processes.
\end{remark}

\begin{proposition}\label{apsProp}
Let $X=(X_t)_{t\in\bR}$ be an almost periodic stationary process. Then the following are true:
\begin{itemize}
    \item[a)] For any $n\in\N$ and $(x_1,\dotso,x_n)\in \R^n$, the function $\mu_{x_1,\dotso,x_n}(t)$ is uniformly continuous in $t\in\R$.
    \item[b)] The set of probability measures $(\mathcal{L}_{X_t})_{t\in\R} \subset \mathcal{P}(\R)$ is relatively compact.
    \item[c)] Let $g:\R\to \mathbb{R}$ be a continuous function.\,Then $g(X_t)$ is an almost periodic stationary process.
    \item[d)] If $X$ is $L^2$-uniformly integrable, then $X$ is almost periodically correlated. 
\end{itemize}
\end{proposition}

\begin{proof}
Statement a) follows from [\ref{Levitan}, Property 2, p.2], and b) follows from Bochner's double sequence criterion, since by choosing $x=0$, $(b_n)_{n\in\N}=0$ we find for a general sequence $(a_n)_{n\in\N}$ a subsequence $(a_n')_{n\in\N}$ such that $(\mathcal{L}(X_{a_n'}))_{n\in\N}$ converges weakly, showing relative compactness of $(\mathcal{L}(X_t))_{t\in\R}$.\\ 
c) This is a direct consequence of Bochner's double sequence criterion and the continuous mapping theorem, see [\ref{Dudley}, Theorem 9.3.7].\\
d) We use Bochner's double sequence criterion.\,At first we know that the limits
\begin{align*}
    \lim_{n\to\infty}\lim_{m\to\infty}\mu_{t+a_{n}+b_m,s+a_{n}+b_m}=\lim_{n\to\infty}\mu_{t+a_{n}+b_n,s+a_{n}+b_n}
\end{align*}
exist in the topology of the weak convergence. By our uniform integrability condition, we conclude that the limits also exist in the Wasserstein space $W_2$. As a consequence we know that the limits $$\lim_{n\to\infty}\lim_{m\to\infty}\bE X_{t+a_n+b_m}X_{s+a_n+b_m}=\lim_{n\to\infty}\bE X_{t+a_n+b_n}X_{s+a_n+b_n}$$ exist.\,By definition it follows that $\bE X_{t+u}X_{s+u}$ is almost periodic in $u$ for every $t,s\in\bR$ (that $u\mapsto \mathbb{E}X_{t+u}X_{s+u}$ is continuous follows similarly from the continuity of $u\mapsto \mu_u(s,t)$ with respect to the Wasserstein metric $W_2$).\,Similarly, $u\mapsto \mathbb{E} X_{t+u}$ is almost periodic for every $t\in\R$. Since sums and products of almost periodic functions are again almost periodic, the claim follows.
\end{proof}

With the metric $\gamma_n$ from Definition \ref{defnewmet} we can easily deduce a simple equivalent condition based on characteristic functions when a stochastic process is almost periodic stationary:

\begin{theorem}\label{thechar}
Let $(X_t)_{t\in\bR}$ be a stochastic process. The process is almost periodic stationary if, and only if, for every $n\in\bN$, $t_1,\dotso,t_n\in\R$, $K\subset \bR^n$ compact and $\varepsilon>0$ there exists an $l>0$ such that for every $a\in\bR$ there exists a $\tau\in [a,a+l]$ with
\begin{align}
    \sup_{x\in\R} \sup_{z\in K}|\hat{\mathcal{L}}(X_{t_1+x},\dotso,X_{t_n+x})(z)-\widehat{\mathcal{L}}(X_{t_1+x+\tau},\dotso,X_{t_n+x+\tau})(z)|<\varepsilon\label{t1}
\end{align}
and the functions
\begin{align}
    \bR\ni x\mapsto \widehat{\mathcal{L}}(X_{t_1+x},\dotso,X_{t_n+x})(z)\label{t2}
\end{align}
are continuous for every fixed $z$ and every vector $(t_1,\dotso,t_n)\in\bR^n$.
\end{theorem}

\begin{proof}
First observe that (\ref{t2}) is equivalent to the continuity of $x\mapsto \mathcal{L}(X_{t_1+x},\dotso,X_{t_n+x})$ by L\'evy's continuity theorem.\,The rest is a consequence of Definition \ref{def1} when choosing the metric $\gamma_n$: assume that the process is almost periodic. For every compact set $K\subset \bR^n$ there exists an $m>0$ such that $K\subset [-m,m]^n$. We see that
\begin{align*}
   \sup_{x\in\R}\sup_{z\in K}|\widehat{\mathcal{L}}(X_{t_1+x},\dotso,X_{t_n+x})(z)-\widehat{\mathcal{L}}(X_{t_1+x+\tau},\dotso,X_{t_n+x+\tau})(z)|\\
   \leq 2^m \sup_{x\in\R}\gamma_n(\mathcal{L}(X_{t_1+x},\dotso,X_{t_n+x}),\mathcal{L}(X_{t_1+x+\tau},\dotso,X_{t_n+x+\tau})),
\end{align*}
showing (\ref{t1}).\\
For the converse direction fix $\varepsilon>0$.\,Observe that by choosing $K=[-m,m]^d$ such that $\sum_{k=m+1}^\infty2^{-k}<\varepsilon/4$ we have that
\begin{align*}
    \sum\limits_{k=1}^\infty& 2^{-k} \|\widehat{\mathcal{L}}(X_{t_1+x},\dotso,X_{t_n+x})-\widehat{\mathcal{L}}(X_{t_1+x+\tau},\dotso,X_{t_n+x+\tau})\|_{C([-k,k]^d)}\\
    \le&  \|\widehat{\mathcal{L}}(X_{t_1+x},\dotso,X_{t_n+x})-\widehat{\mathcal{L}}(X_{t_1+x+\tau},\dotso,X_{t_n+x+\tau})\|_{C([-m,m]^d)}+\varepsilon/2
\end{align*}
for every $x\in\R$. For given $t_1,\dotso,t_n\in\R$ we choose $\tau$ such that $$ \|\widehat{\mathcal{L}}(X_{t_1+x},\dotso,X_{t_n+x})-\widehat{\mathcal{L}}(X_{t_1+x+\tau},\dotso,X_{t_n+x+\tau})\|_{C([-m,m]^d)}<\varepsilon/2 \quad\textrm{ for all } x\in\R$$. From this we get
\begin{align*}
    \sup\limits_{x\in\R}\gamma_n(\mathcal{L}(X_{t_1+x},\dotso,X_{t_n+x}),\mathcal{L}(X_{t_1+x+\tau},\dotso,X_{t_n+x+\tau}))<\varepsilon.
\end{align*}
As $\varepsilon$ was arbitrary, this shows that the process is almost periodic stationary.
\end{proof}

From Theorem \ref{thechar} we see that it is indeed possible to characterize almost periodic stationary processes via the characteristic function. 
In many cases it is much easier to calculate the difference between characteristic functions, e.g.\,in the case of two infinitely divisible distributions, which we consider in the following chapter.

\section{Almost periodic stationarity of infinitely divisible random processes}
We derive sufficient and necessary conditions for an infinitely divisible process $(X_t)_{t\in \bR}$ to be almost periodic stationary, but first we recall the definitions of an infinitely divisible distribution and an infinitely divisible process .

A probability measure $\mu\in\mathcal{P}(\bR^n)$ is \emph{infinitely divisible}, if for every $m\in\N$ there exists $\mu_m\in\mathcal{P}(\bR^n)$ such that $\mu=\mu_m^{*m}$ (the $m$-fold convolution of $\mu_m$ with itself). By the \emph{L\'evy-Khintchine formula} (see e.g. [\ref{Kallenberg}], [\ref{Sato}]), $\mu\in \mathcal{P}(\R^n)$ is infinitely divisible if and only if there exist some $\gamma \in\R^n$, a symmetric positive semi-definite matrix $A\in\R^{n\times n}$ and a measure $\nu$ on $\R^n$ with
\begin{align*}
    \int\limits_{\R^n} \min(1,|x|^2)\,\nu(dx)<\infty \textrm{ and } \nu(\{0\})=0
\end{align*}
($\nu$ is called a \emph{L\'{e}vy measure}) such that
\begin{align}
    \hat{\mu}(z) &= \exp(\psi(z))\label{muLK}
    \intertext{with}
    \psi(z) &=i\langle\gamma, z\rangle-\frac12 \langle z, Az\rangle  +\int\limits_{\bR^n} (e^{i\langle z,y\rangle }-1-i\langle z,y\rangle \mathds{1}_{|x|\le 1})\nu(dy),\quad z\in\R^n\label{levykhin}
\end{align}
(here, $\langle \cdot,\cdot\rangle$ denotes the Euclidean inner product in $\R^n$).\,We call $(A,\gamma,\nu)$ the \emph{characteristic triplet} of $\mu$, which is known to be unique, and $\psi$ the \emph{characteristic exponent} of $\mu$. Further, for every $\gamma\in\R^n$, symmetric positive semi-definite $A\in\R^{n\times n}$ and L\'evy measure $\nu$, the right-hand side of (\ref{muLK}), (\ref{levykhin}) defines the characteristic function of an infinitely divisible distribution. Instead of the indicator function $\mathds{1}_{|x|\le 1}$ in (\ref{levykhin}) one can choose a continuous representation function $c:\R^n\to\R$ which is bounded, with compact support and satisfies $c(x)=1$ for $|x|\le 1$ (e.g. see [\ref{Jacod}, Section VII.2a], [\ref{Sato}, Remark 8.4]) so that the integrability with respect to $\nu$ in (\ref{levykhin}) is still assured. This only changes the value of $\gamma$. In this case we write $(A,\gamma^c,\nu)_c$ for the characteristic triplet of $\mu$.
A stochastic process $(X_t)_{t\in \bR}$ is called an \emph{infinitely divisible stochastic process} if all finite-dimensional distributions $(X_{t_1},\dotso,X_{t_n})$, with $n\in \bN$ and $(t_1,\dotso,t_n)\in \bR^n$, are infinitely divisible.\,We write $(A_{t_1,\dotso,t_n},\gamma^c_{t_1,\dotso,t_n}, \nu_{t_1,\dotso,t_n})_c$ and $(A_{t_1,\dotso,t_n},\gamma_{t_1,\dotso,t_n}, \nu_{t_1,\dotso,t_n})$ when $c(x)=\mathds{1}_{|x|\le 1}$ for the characteristic triplet of $(X_{t_1},\dotso,X_{t_n})$, and $\psi_{t_1,\dotso,t_n}$ for its characteristic exponent.


Our next theorem describes how to express almost periodic stationarity for infinitely divisible processes in terms of their characteristic exponents and triplets. For a column vector $x\in\R^d$ we denote by $x^T$ its transpose and for $a,b\in\R$ we denote by $a\wedge b$ the minimum of $a$ and $b$.
\begin{theorem}\label{theoremID}
Let $(X_t)_{t\in\bR}$ be an infinitely divisible process and for each $d\in\N$ let $c_d:\R^d\to\R$ be a continuous function with compact support such that $c_d(x)=1$ for $|x|\le 1$.\,Denote the characteristic exponents and triplets of $(X_{t_1},\dotso,X_{t_d})$ by $\psi_{t_1,\dotso,t_d}$ and $(A_{t_1,\dotso,t_d},\gamma_{t_1,\dotso,t_d}^{c_d},\nu_{t_1,\dotso,t_d})_{c_d}$, respectively.\,Then the following assertions are equivalent:
\begin{itemize}
    \item[a)] The process $(X_t)_{t\in\bR}$ is almost periodic stationary.
    \item[b)] The functions $\R\to C([-k,k]^d,\C)$, $t\mapsto \psi_{{t_1+t,\dotso,t_d+t}\big|_{[-k,k]^d}}$ (the restriction of $\psi_{t_1+t,\dotso,t_d+t}$ to $[-k,k]^d$) are almost periodic for every $k\in\N$, $d\in\bN$ and $(t_1,\dotso,t_d)\in\R^d$.
    \item[c)] The functions 
    \begin{align}
        \label{condition1}&\bR\ni t\mapsto \gamma^{c_d}_{t_1+t,\dotso,t_d+t},\\
        \label{condition2}&\bR\ni t\mapsto A_{t_1+t,\dotso,t_d+t}+\int\limits_{\bR^d} xx^Tc_d(x)^2\nu_{t_1+t,\dotso,t_d+t}(dx),\\
        \label{condition3}&\bR\ni t \mapsto (|x|^3\wedge 1)\,\nu_{t_1+t,\dotso,t_d+t}(dx)
    \end{align}
    are almost periodic in $\bR^d$, $\bR^{d\times d}$ and $\mathcal{M}(\bR^d)$, respectively, for every $d\in\bN$ and $(t_1,\dotso,t_d)\in\R^d$.
\end{itemize}
\end{theorem}
\begin{remark}
That $((|x|^2\wedge 1)\,\nu_{t_1+t,\dotso,t_d+t}(dx))_{t\in\bR}$ is relatively compact follows immediately from the conditions \eqref{condition2} and \eqref{condition3}, but it is not true that the function $(|x|^2\wedge 1)\,\nu_{t_1+t,\dotso,t_d+t}(dx)$ is almost periodic. The reason is that an almost periodic function needs to be continuous by definition, nevertheless $t\mapsto (|x|^2\wedge 1)\,\nu_{t_1+t,\dotso,t_d+t}(dx)$ is not necessarily continuous, which can be seen from condition \eqref{condition2}. One can overcome this shortcoming in dimension one by adding the Gaussian variances to the L\'evy measures as a delta-distribution at point 0, but a similar approach seems to be difficult in dimension greater than 1 as the Gaussian variance is then a matrix (see [\ref{Jacod}, Remark 2.10, Section VII.2].
\end{remark}
\begin{proof}
We show a)$\implies$ b)$\implies$c)$\implies$a). Moreover, we fix a vector $(t_1,\dotso,t_d)$ and always write $\gamma_t,\nu_t,A_t,\psi_t$ and $\mu_t$ for $\gamma^{c_d}_{t+t_1,\dotso,t+t_d}, \nu_{t+t_1,\dotso,t+t_d}, A_{t+t_1,\dotso,t+t_d},\psi_{t+t_1,\dotso,t+t_d}$ and $\mu_{t+t_1,\dotso,t+t_d}$, respectively.\\
a)$\implies$b). Let $(X_t)_{t\in\R}$ be almost periodic stationary and $t_1,\dotso,t_n\in\R$. If $s_n\to t$ as $n\to\infty$, then $\mathcal{L}(X_{t_1+s_n},\dotso, X_{t_d+s_n})$ converges weakly to $\mathcal{L}(X_{t_1+t},\dotso, X_{t_d+t})$ as $n\to\infty$, hence the characteristic function and hence $\psi_{t_1+s_n,\dotso,t_d+s_n}$ converges locally uniformly to $\psi_{t_1+t,\dotso,t_d+t}$, cf. [\ref{Sato}, Lemma 7.7], showing continuity of $t\mapsto~\psi_{t_1+t,\dotso,t_d+t|_{[-k,k]^d}}$. By Theorem \ref{thechar} we know that the functions $t\mapsto \hat{\mu}_{t|_{[-k,k]^d}}$ are almost periodic in the norm $\|\cdot\|_{C([-k,k]^d)}$ for every $k\in \bN$. Observe that
\begin{align}
   |\hat{\mu}_t(z)-\hat\mu_{t+\tau}(z)|=&|\exp\left(\psi_t(z)\right)-\exp\left(\psi_{t+\tau}(z)\right)|\nonumber\\
   =&|\exp(\psi_{t+\tau}(z))|\cdot|1-\exp\left(\psi_t(z)-\psi_{t+\tau}(z)\right)|\label{eqChar}
\end{align}
and
\begin{align*}
    |\exp(\psi_{t}(z))|=&\exp(\Real\psi_t(z))\\
    =&\exp\left(-\left(\frac{1}{2}\langle z, A_tz\rangle+\int\limits_{\bR^d} (1-\cos(\langle x,z\rangle))\nu_t(dx)\right)\right)\\
    \geq &\exp\left(-\left(\frac{1}{2}\langle z,A_tz\rangle+|z|^2\int\limits_{|x|\le 1} |x|^2\nu_t(dx)+2\nu_t(\{x\in\R^d:|x|>1\})\right)\right)
\end{align*}
(where we used $1-\cos(y)\le y^2$ and the Cauchy-Schwarz inequality).\,Since the set $(\mu_t)_{t\in\R}$ is relatively compact by Proposition \ref{apsProp} b),
\begin{align*}
   \sup_{t\in \bR}\sup_{z\in [-k,k]^d} \frac{1}{2}\langle z, A_tz\rangle+|z|^2\int\limits_{|x|\le 1} |x|^2\nu_t(dx)+2\nu_t(\{x\in\R^d:|x|>1\})<\infty
\end{align*}
for every $k\in\N$ by [\ref{Sato}, Exercise 12.5, p.66].\,The previous estimates then imply that
\begin{align*}
    \inf_{t\in \bR}\inf_{z\in [-k,k]^d}|\exp(\psi_{t}(z))|>0.
\end{align*}
Since $t\mapsto \hat{\mu}_{t|_{[-k,k]^d}}$ is almost periodic, this together with (\ref{eqChar}), implies that for every $\delta>0$ there exists an $L_{\delta}>0$ and for each $a\in\R$ a $\tau=\tau(a,\delta)\in [a,a+L_{\delta}]$ such that $\sup\limits_{t\in\R}\sup\limits_{z\in [-k,k]^d} |1-\exp(\psi_t(z) - \psi_{t+\tau}(z))|<\delta$. Since $\C\ni x+iy\mapsto 1-e^{x+iy}$ is continuous with period $2\pi i$ and zero set $2\pi i \Z$, and since $\lim\limits_{x\to\infty} e^x=\infty$ and $\lim\limits_{x\to -\infty} e^x=0$, $|1-e^{x+iy}|<\delta <1$ implies $x+iy\in \bigcup\limits_{m\in\Z} \{ 2\pi im+w: w\in\C, |w|<\varepsilon \}=: D_{\varepsilon}$ for some $\varepsilon=\varepsilon(\delta)>0$, where $\varepsilon(\delta)\to 0$ as $\delta\to 0$. Hence $\psi_t(z)-\psi_{t+\tau}(z) \in D_{\varepsilon(\delta)}$ for every $z\in [-k,k]^d$ and $t\in\R$, and since $\psi_t-\psi_{t+\tau}$ is continuous with $\psi_t(0)=\psi_{t+\tau}(0)=0$, we conclude
\begin{align*}
    \sup\limits_{t\in\R} \sup\limits_{z\in [-k,k]^d} |\psi_t(z)-\psi_{t+\tau}(z)|<\varepsilon
\end{align*}
when $\varepsilon=\varepsilon(\delta)<\pi$, giving the desired almost periodicity of $t\mapsto \hat{\mu}_{t|_{[-k,k]^d}}$.\\
b) $\implies$ c).\,First observe that continuity of $t\mapsto \psi_{t|_{[-k,k]^d}}$ implies continuity of $t\mapsto \hat{\mu}_{t|_{[-k,k]^d}}$ which by [\ref{Jacod}, Theorem 2.9 and Remark 2.10, Section VII.2] implies continuity of the functions given in (\ref{condition1})--(\ref{condition3}). Let us show their almost periodicity. Let $(a'_n)_{n\in\bN}$, $(b'_m)_{m\in\bN}\subset\bR$ be two sequences. By Bochner's double sequence criterion and a standard diagonal subsequence argument, there exist subsequences $(a_n)_{n\in\bN}$ and $(b_m)_{m\in\bN}$ such that 
\begin{align*}
\lim_{n\to\infty}\lim_{m\to\infty}\psi_{t+a_n+b_m} \textrm{ and } \lim_{n\to\infty}\psi_{t+a_n+b_n}
\end{align*}
exist (in the sense of locally uniform convergence) and
\begin{align*}
    \lim_{n\to\infty}\lim_{m\to\infty}\psi_{t+a_n+b_m}(z)=\lim_{n\to\infty}\psi_{t+a_n+b_n}(z)=:\psi_{t,\infty}(z)
\end{align*}
for every $t\in\bR$ and $z\in\bR^d$. As the set of infinitely divisible distributions is closed with respect to the Prokhorov metric (see [\ref{Sato}, Lemma 7.8]), we know that $\lim_{m\to\infty}\psi_{t+a_n+b_m}$ and $\psi_{t,\infty}$ are again the characteristic exponents of infinitely divisible distributions. Denote the characteristic triplet corresponding to $\psi_{t,\infty}$ by $(\gamma_{t,\infty},A_{t,\infty},\nu_{t,\infty})_c$. Then by [\ref{Jacod}, Theorem 2.9 and Remark 2.10, Section VII.2] 
\begin{align*}
    &\lim_{n\to\infty}\lim_{m\to\infty}\gamma_{t+a_n+b_m}=\lim_{n\to\infty}\gamma_{t+a_n+b_n}=\gamma_{t,\infty},\\
    &\lim_{n\to\infty}\lim_{m\to\infty}(|x|^3\wedge 1)\,\nu_{t+a_n+b_m}(dx)=\lim_{n\to\infty}(|x|^3\wedge 1)\,\nu_{t+a_n+b_n}(dx)=(|x|^3\wedge 1)\,\nu_{t,\infty}(dx),
\end{align*}
where the limits are taken weakly, and
\begin{align*}
    &\lim_{n\to\infty}\lim_{m\to\infty}A_{t+a_n+b_m}+\int\limits_{\bR^d} xx^Tc_d(x)^2\nu_{t+a_n+b_m}(dx)\\
    &=\lim_{n\to\infty}A_{t+a_n+b_n}+\int\limits_{\bR^d} xx^Tc_d(x)^2\nu_{t+a_n+b_n}(dx) =A_{t,\infty}+\int\limits_{\bR^d} xx^T c_d(x) \nu_{t,\infty}(dx).
\end{align*}
Hence, by Bochner's double sequence criterion, the functions $\gamma_t$, $(|x|^3\wedge 1) \nu_t(dx)$ and $A_{t}+\int\limits_{\bR^d} xx^Tc(x)^2\nu_{t}(dx)$ are almost periodic.
\\
c)$\implies$ a). The continuity of $t\mapsto \mu_t$ follows again from [\ref{Jacod}, Theorem 2.9 and Remark 2.10, Section VII.2]. For almost periodicity, we 
show that $t\mapsto \mu_t$ satisfies the double sequence criterion. Therefore, let $(a'_n)$ and $(b'_n)$ be sequences. Then there exist subsequences $(a_n)_{n\in\bN}$ and $(b_n)_{n\in\bN}$ such that 
\begin{align*}
    \lim_{n\to\infty}\lim_{m\to\infty} \gamma_{t+a_n+b_m}&=\lim_{n\to\infty} \gamma_{t+a_n+b_n},\\
     \lim_{n\to\infty}\lim_{m\to\infty} A_{t+a_n+b_m}+\int\limits_{\bR^d} x x^Tc_d(x)^2 \nu_{t+a_n+b_n}(dx)
     &=\\\lim_{n\to\infty} A_{t+a_n+b_n}
     &+\int\limits_{\bR^d} xx^Tc_d(x)^2\nu_{t+a_n+b_n}(dx),\\
     \lim_{n\to\infty}\lim_{m\to\infty}(|x|^3\wedge 1)\nu_{t+a_n+b_m}(dx)&=\lim_{n\to\infty}(|x|^3\wedge 1)\nu_{t+a_n+b_n}(dx).
\end{align*}
We have to show for the subsequences $(a_n)_{n\in\bN}$ and $(b_n)_{n\in\bN}$ that 
\begin{align*}
    \lim_{n\to\infty}\lim_{m\to\infty} \mu_{t+a_n+b_m}=\lim_{n\to\infty}\mu_{t+a_n+b_n}.
\end{align*}
This follows again from [\ref{Jacod}, Theorem 2.9 and Remark 2.10, Section VII.2], if we can show that for a sequence $(\gamma_n)_{n\in\N}\subset \R^d$, a sequence $(\nu_n)_{n\in\N}$ of L\'evy measures and symmetric positive semi-definite $d\times d$-matrices $(A_n)_{n\in\N}$, the conditions
\begin{align}
    &\lim_{n\to\infty}\gamma_n= \gamma,\label{limitGamma}\\
    &\lim_{n\to\infty} A_{n}+\int\limits_{\bR^d} xx^T c(x)^2\nu_n(dx)=A'\label{limitGauss},\\
    &\lim_{n\to\infty} (|x|^3\wedge 1)\nu_n(dx)=(|x|^3\wedge 1)\nu'(dx)\label{limitLM},
\end{align}
for $\gamma\in\R^d$, $A'\in\R^{d\times d}$ and $(|x|^3\wedge 1)\nu'\in \mathcal{M}(\R^d)$ imply that $\nu'$ is a L\'evy measure and that $A'-\int\limits_{\R^d} x x^T c_d(x)\nu'(dx)$ is symmetric and positive semi-definite.\,To see this, observe that (\ref{limitGamma}) and (\ref{limitGauss}) imply boundedness of $(\gamma_n)_{n\in\N}$ and $(A_n)_{n\in\N}$ (observe that $\int\limits_{\R^d} xx^Tc_d(x)\nu_n(dx)$ is positive semi-definite), (\ref{limitGauss}) together with $c_d(x)=1$ for $|x|\le 1$ implies $\sup\limits_{n\in\N} \int\limits_{|x|\le 1} |x|^2 \nu_n(dx)<\infty$, and (\ref{limitLM}) together with the fact that $(|x|^3\wedge 1)\nu'(dx)\in \mathcal{M}(\R^d)$ implies $\sup\limits_{n\in\N} \int\limits_{|x|>1} \nu_n(dx)<\infty$ and $\lim\limits_{k\to\infty}\sup\limits_{n\in\N} \int\limits_{|x|> k} \nu_n(dx) =0$. Hence, the sequence $(\rho_n)_{n\in\N}$ of infinitely divisible distributions with characteristic triplets $(A_n,\gamma_n,\nu_n)$ is tight by [\ref{Sato}, Exercise 12.5], so that a subsequence of $(\rho_{n})_{n\in\N}$ converges to an infinitely divisible distribution with characteristic triplet $(A_{\infty}, \gamma_{\infty},\nu_{\infty})$, say. But then $A_{\infty}+ \int\limits_{\R^d} xx^T c_d(x)\nu_{\infty}(dx)=A'$, $\gamma_{\infty} = \gamma$ and $\nu_{\infty} = \nu'$ by [\ref{Jacod}, Theorem 2.9 and Remark 2.10, Section VII.2], so that $\nu'$ is a L\'evy measure and $A'-\int\limits_{\R^d} xx^Tc_d(x)\nu_{\infty}(dx)$ is positive semi-definite, finishing the proof.
\end{proof}

\section{almost periodic stationarity of $X_t= \int_{\R^d} f(t,s) dL(s)$}
In this section we derive estimates for random variables $X,Y$ of the form 
\begin{align}
    \binom{X}{Y}=\int\limits_{\bR^d}\binom{f(t)}{g(t)}dL(t),\label{2xy}
\end{align}
where $L$ is a L\'evy basis and $f,g:\R^d\to \R^m$ are deterministic functions. In particular, we apply these estimates to obtain conditions for stochastic processes $(X_t)_{t\in\R}$ of the form $$X_t=\int\limits_{\R^d} f(t,s)dL(s),\quad t\in\R,$$ to be almost periodic stationary.\\ 
We denote by 
$\mathcal{B}_b(\R^d)$ the set of all bounded Borel sets and by $\lambda^d$ the Lebesgue measure on $\R^d$.\,We recall the definition of a L\'evy basis and the Rajput-Rosinski exponent $\Psi$ of a L\'evy basis 
in order to obtain the above mentioned estimates.

\begin{definition}
A \emph{L\'evy basis} on $\R^d$ is a family $(L(A))_{A\in \mathcal{B}_b(\R^d)}$ of real valued random variables such that
\begin{itemize}
    \item[i)] $L\left( \bigcup_{n\in\N_0} A_n \right) = \sum_{n\in\N_0} L(A_n)$ a.s. for pairwise disjoint sets $(A_n)_{n\in\N_0}\subset  \mathcal{B}_b(\R^d)$ with $\bigcup_{n\in\N_0} A_n \in \mathcal{B}_b(\R^d)$.
    \item[ii)] $L(A_i)$ are independent for pairwise disjoint sets $A_1,\dots,A_n \in  \mathcal{B}_b(\R^d)$ for every $n\in\N$.
    \item[iii)] There exists $a\in [0,\infty)$, $\gamma \in \R$ and a L\'evy measure $\nu $ on $\R$ such that 
\begin{align*}
    \widehat{L(A)}(z)=\mathbb{E} e^{izL(A)} = \exp(\psi(z) \lambda^d(A) ),\quad A\in  \mathcal{B}_b(\R^d),\, z\in\R,
\end{align*}
where 
\begin{align*}
    \psi_L(z) = i\gamma z - \frac{1}{2}a z^2 + \int_{\R}(e^{ixz}-1-ixz\mathds{1}_{[-1,1]}(x) )\nu(dx),\quad z\in\R.
\end{align*}
The triplet $(a,\gamma,\nu)$ is called \emph{characteristic triplet} of $L$ and $\psi_L$ its \emph{characteristic exponent}.
\end{itemize}
\end{definition}





Let $L$ be a L\'evy basis on $\R^d$ with characteristic triplet $(a,\gamma,\nu)$ and let $f:\R^d\to \R$ be a Borel measurable function. Then $f$ is in the \emph{domain of $L$} and we write $f\in D(L)$, if the integral $\int\limits_{\R^d} f(s)\,dL(s)$ can be defined as a limit in probability of integrals of suitable simple functions as outlined by Rajput and Rosinski in [\ref{Rajput}, p. 460] and we say then that the integral exists in the sense of Rajput and Rosinski. By [\ref{Rajput},~Theorem~2.7], $f\in D(L)$ if and only if
\begin{align}
    \int\limits_{\R^d} \Psi(f(t))\, dt<\infty,\label{RajputChar}
\end{align}
where 
\begin{align*}
    \Psi(z):= | U(z) | + a z^2  +V(z), \quad z\in\R,
\end{align*}
is the \emph{Rajput-Rosinski exponent of $L$} with
\begin{align*}
    U(z):&=  \gamma z   +\int\limits_{\R}sz (\mathds{1}_{| sz|\le 1 } - \mathds{1}_{|s|\le 1} ) \,\nu(ds),\\
    V(z):&= \int\limits_{\R} \min( 1, s^2z^2) \,\nu(ds).
\end{align*}
%
If $f=(f_1,\dotso,f_m)^T$ with $f_1,\dotso,f_m\in D(L)$, then $$\int\limits_{\R^d} f(s)\,dL(s)=\left(\int\limits_{\R^d} f_1(s)\,dL(s),\dotso, \int\limits_{\R^d} f_m(s)\,dL(s)\right)^T$$ is infinitely divisible with characteristic function
\begin{align}
    \exp\left( \,\,\int\limits_{\R^d}\psi_L(z^Tf(t))\,dt \right),\quad z\in\R^m.\label{charIntegral2}
\end{align}

\begin{definition}
For a measurable function $f:\R^d\to\R$ we define the \emph{distribution function} of $f$ as
\begin{align*}
    d_f(\alpha)=\lambda^d(\{ x\in \R^d: |f(x)|>\alpha \}) \textrm{,  }\alpha>0.
\end{align*}
\end{definition}
We introduce a short notation for functions $f:\bR^d\to\bR^m$, which are in a sense good integrands for the L\'{e}vy basis $L$. 
\begin{definition}We say that $f=(f_1,\dotso,f_m)^T\in B(L,\bR^m)$, if the integral
\begin{align*}
    \int\limits_{|r|>1}|r|\int\limits_{0}^{1/|r|} d_{f_k}(\alpha)\,d\alpha\,\nu(dr)
\end{align*}
is finite for every $k\in \{1,\dotso,m\}$. If the dimension $m$ is clear, we simply write $B(L)$.
\end{definition}
In the next theorem, we discuss the distance between two characteristic exponents of stochatic integrals with different kernels. Later, we use this theorem to derive sufficient conditions for almost periodicity of processes defined by stochastic integrals.\\ In \cite[Proposition 5.2]{Berger} it was shown that $B(L,\R)\cap L^1(\bR^d,\R)\cap L^2(\bR^d,\R)\subset D(L)$ (more precisely, all components of $f$ are in $D(L)$).\,We denote by $f^+$ and $f^-$ the positive and negative part of a real-valued function $f$, respectively. 
\begin{theorem}\label{theoremloolio}
Let $L$ be a L\'evy basis on $\R^d$ with characteristic triplet $(a,\gamma,\nu)$ and let $(X,Y)\in \R^{n\times n}$ be a $2n$-dimensional infinitely divisible random vector given by $$\binom{X}{Y}=\int\limits_{\bR^d}\binom{f(t)}{g(t)}dL(t),$$ 
where $f,g\in L^1(\R^d,\R^n)\cap L^2(\R^d,\R^n)\cap B(L)$. Let $\psi_f$ and $\psi_g$ be the characteristic exponents of $X$ and $Y$, respectively. Then 
\begin{align*}
 &\left|\psi_f(z)-\psi_g(z)\right|\\
    &\le\int\limits_{|r|\le 1} r^2\nu(dr)\left(\int\limits_0^\infty \alpha(|d_{(z^Tf)^+}(\alpha)-d_{(z^Tg)^+}(\alpha)|+|d_{(z^Tf)^-}(\alpha)-d_{(z^Tg)^-}(\alpha)|)\,d\alpha\right)\\
    &\,\,+\frac{a}{2} \left|\,\int\limits_{\R^d} (z^Tf(x))^2dx-\int\limits_{\R^d} (z^Tg(x))^2dx\right|+\left|\gamma z^T\left(\,\,\int\limits_{\bR^d} f(x)dx-\int\limits_{\bR^d}g(x)dx\right)\right|\\
    &\,\,+\int\limits_{1<|r|\le R} |r|\,\nu(dr)\left(\int_0^{\infty} |d_{(z^Tf)^+}(\alpha)-d_{(z^Tg)^+}(\alpha)|+ |d_{(z^Tf)^-}(\alpha)-d_{(z^Tg)^-}(\alpha)|\,d\alpha\right)\\
    &\,\,+3\int\limits_{|r|>R}|r| \int_{0}^{1/|r|} d_{z^Tf}(\alpha)\,d\alpha \,\nu(dr)+3\int\limits_{|r|>R}|r| \int_{0}^{1/|r|} d_{z^Tg}(\alpha)\,d\alpha \,\nu(dr)
\end{align*}
for every $z\in \bR^n$ and $R>1$.
\end{theorem}
\begin{proof}
By (\ref{charIntegral2}) the characteristic exponents of $X$ and $Y$ can be represented by $\psi_f(z)=\int\limits_{\bR^d} \psi_L( z^T f(y))dy$ and $\psi_g=\int\limits_{\bR^d} \psi_L( z^T g(y))dy$.\,We split $\psi_L$ into four parts
\begin{align*}
   \psi_L(z)=\psi_{1}(w)+\psi_{2,R}(w)+\psi_{3,R}+\psi_4(w)\textrm{, }R>1,\, w\in\R,
\end{align*}
where 
\begin{align*}
    \psi_1(w)&:=\int\limits_{|r|\le 1} (e^{irw}-1-irw) \nu(dr),\\
    \psi_{2,R}(w)&:=\int\limits_{1<|r|\le R} (e^{irw}-1) \nu(dr),\\
    \psi_{3,R}(w)&:=\int\limits_{|r|>R} (e^{irw}-1) \nu(dr)\textrm{ and }\\
    \psi_4(w)&:=-\frac{a}{2}  w^2+i\gamma w.
\end{align*}
Denoting by $L_1$, $L_{2,R}$, $L_{3,R}$ and $L_{4}$ the L\'evy bases with the corresponding exponents, it is easily seen that also $f,g\in B(L_1)\cap B(L_{2,R})\cap B(L_{3,R})\cap B(L_{4})$ and hence $f$ and $g$ are in the corresponding domains.
Clearly, $\psi_1,\psi_{2,R}\in C^\infty(\bR)$ and
\begin{align*}
    \psi_1'(w)&=\int\limits_{|r|\le 1} ir(e^{irw}-1)\nu(dr),\\
    \psi_1''(w)&=-\int\limits_{|r|\le 1} e^{irw} r^2\nu(dr),\\
    \psi_{2,R}'(w)&=i\int\limits_{1<|r|\le R} e^{irw} r\nu(dr).
\end{align*}
Assume for the moment that $z^Tf\ge 0$, then we know
\begin{align*}
    &\int\limits_{\bR^d} \psi_1(z^Tf(x))dx+\int\limits_{\bR^d} \psi_{2,R}(z^Tf(x))dx\\
    =&\int\limits_{\bR^d} \int_0^{z^Tf(x)}\int_0^y\psi_1''(\alpha)d\alpha\,dy\,dx+\int\limits_{\bR^d}\int_0^{z^Tf(x)} \psi'_{2,R}(y)\,dy\,dx\\
    =& \int_0^{\infty} d_{z^Tf}(y)\int_0^y\psi_1''(\alpha)\,d\alpha\,dy+\int_0^{\infty}d_{z^Tf}(y) \psi'_{2,R}(y)\,dy,
\end{align*}
where we used that $f\in L^1(\R^d,\R^n)\cap L^2(\R^d,\R^n)$, the boundedness of $\psi_1''$ and $\psi'_{2,R}$ and [\ref{Grafakos}, Proposition 1.1.4] in order to apply Fubini's theorem.\,A similar reasoning when $z^Tf\le 0$ and decomposing $z^Tf$ into positive and negative parts then leads to 
\begin{align}
    \psi_f(z) =& \int\limits_0^{\infty} d_{(z^Tf)^+}(y)\int_0^y\psi_1''(\alpha)\,d\alpha\,dy+\int\limits_0^{\infty} d_{(z^Tf)^-}(y)\int_0^{-y}\psi_1''(\alpha)\,d\alpha\,dy\label{esPsif}\\
    &+\int_0^{\infty}(d_{(z^Tf)^+}(y) \psi'_{2,R}(y)+d_{(z^Tf)^-}(y) \psi'_{2,R}(-y))\,dy-\frac{a}{2}\int\limits_{\R^d} (z^Tf(x))^2 dx\nonumber\\
    &+ i\gamma z^T\int\limits_{\R^d} f(x)dx+\int\limits_{\bR^d} \psi_{3,R}(z^Tf(x))dx.\nonumber
\end{align}
A similar formula holds for $\psi_g(z)$. Since $|\psi_1''(\alpha)|\le \int_{|r|\le 1} r^2\nu(dr)$ and $|\psi_{2,R}'(\alpha)|\le~\int_{1<|r|\le R} |r|\nu(dr)$, the claim will follow if we can show 
\begin{align*}
    &\left|\int\limits_{\bR^d} \psi_{3,R}(z^Tf(x))dx-\int\limits_{\bR^d} \psi_{3,R}(z^Tg(x))dx\right|\\
    &\le 3\int\limits_{|r|>R}|r| \int_{0}^{1/|r|} d_{z^Tf}(\alpha)d\alpha \nu(dr)+3\int\limits_{|r|>R}|r| \int_{0}^{1/|r|} d_{z^Tg}(\alpha)d\alpha \nu(dr).
\end{align*}
To see this, we observe 
\begin{align*}
    \int\limits_{|r|>R}\int\limits_{\bR^d} (e^{ir z^Tf(x)}-1)dx \nu(dr)=& \int\limits_{|r|>R}\,\,\int\limits_{|r z^Tf(x)|\le 1} (e^{ir z^Tf(x)}-1)dx \nu(dr)\\
    &+ \int\limits_{|r|>R}\,\,\int\limits_{|rz^Tf(x)|> 1} (e^{ir z^Tf(x)}-1)dx \nu(dr)
\end{align*}
and it follows
\begin{align}
\left|\,\int\limits_{|r|>R}\int\limits_{\bR^d} (e^{ir z^Tf(x)}-1)dx\nu(dr)\right|\le& \int\limits_{|r|>R}|r| \int\limits_{|r z^Tf(x)|\le 1} |z^Tf(x)|dx\nu(dr)\label{lstterm}\\
&+2 \int\limits_{|r|>R}d_{z^Tf}\left(\frac{1}{|r|}\right)\nu(dr)\nonumber\\
\le&3\int\limits_{|r|>R}|r| \int_{0}^{1/|r|} d_{z^Tf}(\alpha)d\alpha \nu(dr),\nonumber
\end{align}
using that $|e^{ix}-1|= |ix\int_0^1 e^{itx}dt | \le |x|$ for $x\in\R$, and applying [\ref{Grafakos}, Exercise~1.1.10,~ p.14] on the first term. An application of Fubini's theorem finishes the proof.
\end{proof}
As we now obtained for processes $(X,Y)$ of the form (\ref{2xy}) an estimate for the difference of their characteristic functions in terms of the distribution functions $d_f, d_g$, we can derive a sufficient condition when processes $(X_t)_{t\in\R}$ given by $X_t~=~\int_{\R^d} f(t,s)dL(s)$ with $f(t,\cdot)\in L^1(\R^d,\R)\cap L^2(\R^d,\R)$ are almost periodic stationary.\,To do so, we need the following lemma.  
\begin{lemma}\label{lemmaloolio}
Let $p\in [1,\infty)$ and $f,g\in L^p(\bR^d,\R)$. Denote by $$\prod(f,g):=\{(\bar{f},\bar{g})\in L^p(\bR^d,\R)\times L^p(\bR^d,\R):d_f=d_{\bar{f}}\textrm{ and }d_g=d_{\bar{g}}\}.$$ Then
\begin{align*}
    \int\limits_{0}^\infty \alpha^{p-1} |d_f(\alpha)-d_g(\alpha)|d\alpha\le (\|f\|_{L^p}^{p-1}+\|g\|_{L^p}^{p-1}) \inf_{(\bar{f},\bar{g})\in \prod(f,g)}\|\bar{f}-\bar{g}\|_{L^p}.
\end{align*}
\end{lemma}
\begin{proof}
Let $(\bar{f},\bar{g})\in \prod(f,g)$. We see
\begin{align*}
    p\int\limits_{0}^\infty \alpha^{p-1} |d_f(\alpha)-d_g(\alpha)|d\alpha&= p\int\limits_{0}^\infty \alpha^{p-1} |d_{\bar{f}}(\alpha)-d_{\bar{g}}(\alpha)|d\alpha\\
    &= p\int\limits_{0}^\infty \alpha^{p-1} \left| \int\limits_{\bR^d} \left(\mathds{1}_{|\bar{f}(x)|>\alpha}(x)-\mathds{1}_{|\bar{g}(x)|>\alpha}(x)\right) dx\right|d\alpha.
\end{align*}
Since
\begin{align*}
    |\mathds{1}_{|\bar{f}(x)|>\alpha}(x)-\mathds{1}_{|\bar{g}(x)|>\alpha}(x)|=\mathds{1}_{|\bar{f}(x)|\le\alpha< |\bar{g}(x)|}(x)+\mathds{1}_{|\bar{g}(x)|\le\alpha<|\bar{f}(x)|}(x)
\end{align*}
we obtain by the triangular inequality
\begin{align*}
    &p\int\limits_{0}^\infty \alpha^{p-1} |d_f(\alpha)-d_g(\alpha)|d\alpha \\
    \le & p \int\limits_{\bR^d}\int\limits_{0}^\infty\mathds{1}_{|\bar{f}(x)|\le\alpha< |\bar{g}(x)|} \alpha^{p-1}d\alpha \,dx+p\int\limits_{\bR^d}\int\limits_{0}^\infty\mathds{1}_{|\bar{g}(x)|\le\alpha< |\bar{f}(x)|} \alpha^{p-1} d\alpha\, dx\\
    =&p\int\limits_{\bR^d} \mathds{1}_{|\bar{f}(x)|<|\bar{g}(x)|}\int\limits_{|\bar{f}(x)|}^{|\bar{g}(x)|}\alpha^{p-1} d\alpha\, dx+p\int\limits_{\bR^d} \mathds{1}_{|\bar{g}(x)|<|\bar{f}(x)|}\int\limits_{|\bar{g}(x)|}^{|\bar{f}(x)|}\alpha^{p-1} d\alpha \,dx\\
    =&\int\limits_{\bR^d} \mathds{1}_{|\bar{f}(x)|<|\bar{g}(x)|}(|\bar{g}(x)|^p-|\bar{f}(x)|^p) dx+\int\limits_{\bR^d} \mathds{1}_{|\bar{g}(x)|<|\bar{f}(x)|}(|\bar{f}(x)|^p-|\bar{g}(x)|^p) dx\\
    =&\int\limits_{\bR^d} \big| |\bar{f}(x)|^p-|\bar{g}(x)|^p\big| dx.
    \end{align*}
    We know that $\big||a|^p-|b|^p\big|\le p(|a|^{p-1}+|b|^{p-1})\big||a|-|b|\big|$ for $p>1$ and for $p=1$ we use that $\big||a|-|b|\big|\le |a-b|$ for every $a,b\in\bR$. Hence, we have
    \begin{align*}
    \int\limits_{0}^\infty \alpha^{p-1} |d_f(\alpha)-d_g(\alpha)|d\alpha\le&  \int\limits_{\bR} |\bar{f}(x)-\bar{g}(x)|(|\bar{f}(x)|^{p-1}+|\bar{g}(x)|^{p-1})dx\\
    \le& \|\bar{f}-\bar{g}\|_{L^p}(\|f\|_{L^p}^{p-1}+\|g\|_{L^p}^{p-1}),
\end{align*}
where we used in the last line that $\|f\|_{L^p}=\|\bar{f}\|_{L^p}$ and the H\"older inequality with exponent $p$ (if $p>1$).
\end{proof}

\begin{corollary}\label{a10} 
Let $f:\R\times \R^d\to\R$ be a measurable function with $f(t,\cdot)\in L^1(\R^d,\R)\cap L^2(\R^d,\R)$ such that the function $T_f:\R\to L^1(\R^d,\R)\cap L^2(\R^d,\R)$ given by $T_f(t):=f(t,\cdot)$ is continuous in $L^1(\R^d,\R)$ and $L^2(\R^d,\R)$.\,Furthermore, let $L$ be a L\'evy basis with characteristic triplet $(a,\gamma,\nu)$ such that 
\begin{align*}
    \int\limits_{|r|>1} |r| \sup\limits_{t\in\R} \int\limits_{0}^{1/|r|} d_{f(t,\cdot)}\left(\alpha\right)\,d\alpha \,\nu(dr)< \infty.
\end{align*}
Then $X_t:=\int_{\bR^d}f(t,s)dL(s)$ is almost periodic stationary if for every $\varepsilon>0$ there exist some $L_{\varepsilon}$ and $\tau\in[a,a+L_{\varepsilon}]$ as well as $s_1(\tau),s_2(\tau) \in\R$ for all $a\in\R$ such that
\begin{align}
    \sup\limits_{t\in\R}\|f(t,\cdot)-f(t+\tau,\cdot+s_1(\tau))\|_{L^1}+\sup\limits_{t\in\R}\|f(t,\cdot)-f(t+\tau,\cdot+s_2(\tau))\|_{L^2} <\varepsilon.\label{apstationarylp}
\end{align}
\end{corollary}
\begin{proof}
From [\ref{BergerMohamed}, Proposition 3] it follows that $(X_t)_{t\in\R}$ is in the domain of $L$ and that $(X_t)_{t\in\R}$ is stochastically continuous. Next, observe that the continuity of $t\mapsto T_f(t)$ in $L^p(\R^d,\R)$ implies boundedness of $\|T_f(t)\|_{L^p}$ for $t$ in compacts, and using (\ref{apstationarylp}), we get $\sup_{t\in\R} \|T_f(t)\|_{L^p}<\infty$ for $p=1,\,2$.\,The rest follows directly from Theorem~\ref{theoremID}~b), Theorem \ref{theoremloolio} and Lemma~\ref{lemmaloolio}, by observing that for fixed $t_1,\dotso,t_d\in\R$, and $z\in [-k,k]^d$, $k\in\N$,
\begin{align*}
    &\bigg( \big(f(t+t_1+\tau,\cdot),\dotso,f(t+t_d+\tau,\cdot )\big)z\bigg)^+\\
\intertext{and}
    &\bigg( \big(f(t+t_1+\tau,\cdot+s_i(\tau)),\dotso,f(t+t_d+\tau,\cdot +s_i(\tau))\big)z\bigg)^+,\, i=1,2,
\end{align*}
have the same distribution function and choosing $R>1$ large enough such that 
\begin{align*}
    \sup\limits_{z\in [-k,k]^d} \int\limits_{|r|>R} |r| \sup\limits_{t\in\R}\int\limits_{0}^{1/|r|} d_{\left( f(t+t_1,\cdot),\dotso, f(t+t_d,\cdot) \right)z}(\alpha)d\alpha\nu(dr)
\end{align*}
becomes sufficiently small.
\end{proof}
A canonical choice for $s_1(\tau)$ and $s_2(\tau)$ in (\ref{apstationarylp}) is $s_1(\tau)=s_2(\tau)=0$, which leads to almost periodicity in probability, as shown in Theorem \ref{a1} below. Another canonical choice is $s_1(\tau)=s_2(\tau)=\tau$. While the choice $s_1(\tau)=s_2(\tau)=0$ is not possible to ensure almost periodic stationarity of Ornstein-Uhlenbeck processes under suitable conditions, the choice $s_1(\tau)=s_2(\tau)=\tau$ works for this class as will be shown in Section 6.

\section{Almost periodicity in probability}
In the following, we take up the definition of stochastic processes that are almost periodic in probability (see [\ref{Bedouhene}], [\ref{Precupanu}]).\,Again, the definition is independent of the choice of the metric as long as it generates the same topology.\,We introduce the Ky-Fan metric, which induces the topology of convergence in probability (see [\ref{Dudley}, p. 289]) and derive conditions ensuring that processes of the form $X_t=\int_{\R^d} f(t,s)dL(s)$ are almost periodic in probability by estimating the Ky-Fan metric.\,Further, we state examples for such processes $(X_t)_{t\in\R}$.\,Let $(\Omega, \mathcal{F}, P)$ be a probability space and let $L^0(\Omega,\R)$ be the space of measurable functions from $\Omega$ to $\R$. We identify two random variables which coincide $P$-almost surely.\,Observe that the Ky-Fan metric is in general not separable (see [\ref{Doob},~Section IV.3,~Theorem] for sufficient conditions for separability),\,but as mentioned in Remark \ref{Rem <3}, separability is not needed for our considerations. Specialising Definition \ref{DefiAPD} now to this setting, we have:




\begin{definition}
Let $d$ be a metric on $L^0(\Omega,\R)$ which induces convergence in probability. A stochastically continuous process $X=(X_t)_{t\in\R}$ is \emph{almost periodic in probability} if the function $\bR\to L^0(\Omega,\R)$, $t\mapsto X_t$, is almost periodic with respect to the metric $d$. 
\end{definition}

Again, the specific form of the metric $d$ is irrelevant as long as it induces convergence in probability. Typical examples are 
\begin{align*}
    d(X,Y)=\mathbb{E}\left( \frac{|X-Y|}{1+|X-Y|} \right),\\
    d(X,Y)=\mathbb{E}\min(1,|X-Y|),
\end{align*}
or the Ky-Fan metric $\alpha$, defined by 
\begin{align*}
    \alpha(X,Y):=\inf\{\varepsilon\ge 0: P(|X-Y|>\varepsilon)\le \varepsilon \}
\end{align*}
(see [\ref{Dudley}, p. 289]). It is known that $X=(X_t)_{t\in\R}$ is almost periodic in probability if and only if $X$ is stochastically continuous and if for any $\varepsilon>0$ and $\eta>0$, there exists $L_{\varepsilon,\eta}>0$ such that any interval of length $L_{\varepsilon,\eta}$ contains at least a $\tau\in\R$ for which
\begin{align*}
    P(|X_t-X_{t+\tau}|>\eta)\le \varepsilon \quad \textrm{ for all } t\in\R
\end{align*}
(see  [\ref{Bedouhene}, p. 328], [\ref{Precupanu}, Definition 2.3 and Remark 2.6 I)]). As expected, almost periodicity in probability implies almost periodic stationarity, see [\ref{Bedouhene}, Theorem 2.14]. 

Next, we prove estimates of the Ky-Fan metric $\alpha(X,Y)$, where $X,Y$ are defined as in (\ref{2xy}), in terms of the characteristic triplet of $L$, $\|f-g\|_{L^1}$, $\|f-g\|_{L^2}$ and the distribution functions of $f$ and $g$.


\begin{lemma}\label{d1Lemma}
Let $L$ be a L\'evy basis on $\R^d$ with characteristic triplet $(a,\gamma,\nu)$ and let $(X,Y)$ be a two-dimensional infinitely divisible random vector given by $$\binom{X}{Y}=\int\limits_{\bR^d}\binom{f(t)}{g(t)}dL(t),$$ where $f,g:\R^d\to\R $ are in $D(L)$.
\begin{itemize}
\item[(i)] If $f,g\in L^2(\bR^d,\R)\cap D(L)$ and $L$ has finite second and vanishing first moment, i.e. $\sigma^2:= \mathbb{E}L([0,1]^d)^2<\infty$ and $\mathbb{E}L([0,1]^d)=0$, then \begin{align*}
        \alpha(X,Y)\le \min\left(1, \left(\sigma^2 \|f-g\|^2_{L^2}\right)^{\frac{1}{3}}\right).
    \end{align*}
\item[(ii)] Let $f,g\in D(L)$ with $f,g\in L^1(\R^d,\R)\cap L^2(\R^d,\R) $.\,Then
\begin{align*}
    \alpha(X,Y)\le\min\left( 1,(14\cdot I_R(f-g))^{\frac{1}{3}}  \right),
\end{align*}
for every $R>1$, where 
\begin{align*}
    I_{R}(f-g):=&  \frac{1}{2}\|f-g\|_{L^2}^2\int\limits_{|r|\le 1} r^2\nu(dr)  + \frac{a}{2} \|f-g\|_{L^2}^2 + |\gamma|\,\|f-g\|_{L^1}\\
    &+ \,\|f-g\|_{L^1}\int\limits_{1<|r|\le R} |r|\nu(dr) + 4\int\limits_{|r|>R}|r|\int\limits_{0}^{1/|r|} (d_{f}(\alpha)+d_{g}(\alpha)) d\alpha\,\nu(dr).
\end{align*}
\end{itemize}
\end{lemma}

\begin{proof}
If the assumptions in (i) hold we have $\mathbb{E}(X-Y)=0$ and Var$(X-Y)=\sigma^2 \|f-g\|_{L^2}^2$ and therefore obtain with Chebyshev's inequality 
\begin{align*}
        \alpha(X,Y)\le  \min\left(1, \left(\sigma^2 \|f-g\|^2_{L^2}\right)^{\frac{1}{3}}\right) .
\end{align*}
In order to show (ii) we first calculate with [\ref{Sasvari}, Lemma 1.6.2] for $\delta>0$,
\begin{align}
P(|X-Y|>\delta)\le 7  \cdot\frac{\delta}{2} \int\limits_{-1/\delta}^{1/\delta}(1-\Re(\varphi_{X-Y}(z)))\,dz \le 7 \cdot\frac{\delta}{2} \int\limits_{-1/\delta}^{1/\delta}|\psi_{f-g}(z)|\,dz\label{KyFanEst},
\end{align}
where $\psi_{f-g}$ is the characteristic exponent of the infinitely divisible random variable $X-Y$ and $\Re(\cdot)$ denotes the real part of a complex number; the last inequality follows from  
\begin{align*}
    |1-e^{\psi(z)}|= |\psi(z)| \left|\int\limits_0^1 e^{t\psi(z)}\,dt\right|\le  |\psi(z)| \int\limits_0^1\left| e^{t\psi(z)}\right| \,dt \le  |\psi(z)|,
\end{align*}
where we used that $z\mapsto e^{t\psi(z)}$ is a characteristic function (of an infinitely divisible distribution), hence $|e^{t\psi(z)}|\le 1$. The proof of Theorem \ref{theoremloolio} (compare (\ref{esPsif}) and its subsequent estimates) gives us for $R>1$ and $z\in\R$\newpage
\begin{align*}
    &|\psi_{f-g}(z)|\\
    &\le \int\limits_{|r|\le 1} r^2\nu(dr) \int\limits_0^{\infty}\alpha (d_{z(f-g)^+}(\alpha) + d_{z(f-g)^-}(\alpha))d\alpha + \frac{a}{2}|z|^2 \|f-g\|_{L^2}^2 + |\gamma||z|\|f-g\|_{L^1}\\
    &+ \int\limits_{1<|r|\le R} |r|\nu(dr) \int\limits_{0}^{\infty} (d_{z(f-g)^+}(\alpha) + d_{z(f-g)^-}(\alpha))d\alpha \\
    &+ \left|\,\,\int\limits_{|r|>R}\int\limits_{\R^d} \left( e^{irz(f(x)-g(x))} -1\right)dx\,\nu(dr)\right|.
\end{align*}
Using that
\begin{align}
    &\left|\,\,\int\limits_{|r|>R}\int\limits_{\R^d} \left( e^{irz(f(x)-g(x))} -1\right)dx\,\nu(dr)\right|\label{eqt2}\\
    &\le \int\limits_{|r|>R} \left( |r| \int\limits_{|r(f(x)-g(x))|\le 1} |z| |f(x)-g(x)|dx +  \int\limits_{|r(f(x)-g(x))|> 1} 2dx \right)\nu(dr)\nonumber\\
    &\le (|z|+2) \int\limits_{|r|>R} |r| \int\limits_0^{1/|r|} d_{f-g}(\alpha)d\alpha \nu(dr)\nonumber,  
\end{align}
similar to (\ref{lstterm}), and $d_{f-g}(\alpha)\le d_f(\frac{\alpha}{2})+d_g(\frac{\alpha}{2})$ (see [\ref{Grafakos}, Proposition 1.1.3]), we hence obtain with [\ref{Grafakos}, Proposition 1.1.4]
\begin{align*}
     &|\psi_{f-g}(z)|\\
     &\le 
    \,\frac{1}{2}|z|^2\int\limits_{|r|\le 1} r^2\nu(dr)  \|f-g\|_{L^2}^2 + \frac{a}{2}|z|^2 \|f-g\|_{L^2}^2 + |\gamma||z|\,\|f-g\|_{L^1}\\
    &\,\,+ \,|z|\,\|f-g\|_{L^1}\int\limits_{1<|r|\le R} |r|\nu(dr) + (|z|+2)\int\limits_{|r|>R}|r|\int\limits_{0}^{1/|r|}\left( d_{f}\left(\frac{\alpha}{2}\right)+d_{g}\left(\frac{\alpha}{2}\right) \right)d\alpha\,\nu(dr)\\
    &\le (1+\max(|z|^2,|z|)) I_R(f-g),
\end{align*}
since $\int\limits_0^{1/|r|} d_f(\frac{\alpha}{2})d\alpha \le 2 \int\limits_0^{1/|r|} d_f(\alpha)d\alpha $. Using 
\begin{align*}
    7\,\frac{\delta}{2} \int\limits_{-1/\delta}^{1/\delta} (1+\max(|z|^2,|z|)) dz \le 7\left(1+\frac{1}{\delta^2}\right) \quad \textrm{ for } \delta\in (0,1],
\end{align*}
we obtain from (\ref{KyFanEst})
\begin{align*}
    \alpha(X,Y)&\le  \inf\bigg\{ \delta\in [0,1]: 7\cdot I_R(f-g) \le \frac{\delta^3}{(\delta^2+1)} \bigg\}\\
    &\le \inf\bigg\{\delta\in (0,1]:14\cdot I_R(f-g)\le \delta^3\bigg\}
\end{align*}
and hence $\alpha(X,Y) \le\min\left( 1,(14\cdot I_R(f-g))^{\frac{1}{3}}  \right)$, finishing the proof.
\end{proof}

\begin{remark}
By assuming that the L\'evy basis $L$ has finite first moment, i.e. $\int_{|r|>1} |r| \,\nu(dr)<~\infty$, the term
$\int_{|r|>R} |r| \int_0^{1/|r|} d_{f-g}(\alpha)d\alpha \nu(dr) $ in (\ref{eqt2}) can be estimated by
\begin{align*}
    \int\limits_{|r|>1} |r| \int\limits_0^{\infty} d_{f-g}(\alpha)d\alpha \nu(dr) = \|f-g\|_{L^1} \int\limits_{|r|>1} |r| \nu(dr),
\end{align*}
hence we also obtain the estimate 
\begin{align*}
    \alpha(X,Y)\le\min\left( 1,(14\cdot \tilde{I}_R(f-g))^{\frac{1}{3}}  \right),
\end{align*}
for every $R>1$, where 
\begin{align*}
    \tilde{I}_{R}(f-g):=&  \frac{1}{2}\|f-g\|_{L^2}^2\int\limits_{|r|\le 1} r^2\nu(dr)  + \frac{a}{2} \|f-g\|_{L^2}^2 + |\gamma|\,\|f-g\|_{L^1}\\
    &+ \,\|f-g\|_{L^1}\int\limits_{1<|r|\le R} |r|\nu(dr) +2 \,\|f-g\|_{L^1} \int\limits_{|r|>1} |r| \nu(dr)
\end{align*}
for $f,g\in D(L)\cap L^1(\R^d,\R)\cap L^2(\R^d,\R)$.
\end{remark}

As we have proven an estimate for the Ky-Fan metric, we can now give a condition when stochastic integrals are indeed almost periodic in probability.

\begin{theorem}\label{a1} 
Let $f:\R\times \R^d\to\R$ be a measurable function with $f(t,\cdot)\in L^1(\R^d,\R)\cap L^2(\R^d,\R)$ such that the function $T_f:\R\to L^1(\R^d,\R)\cap L^2(\R^d,\R)$ given by $T_f(t):=f(t,\cdot)$ is continuous in $L^1(\R^d,\R)$ and $L^2(\R^d,\R)$.\,Furthermore, let $L$ be a L\'evy basis with characteristic triplet $(a,\gamma,\nu)$ such that 
\begin{align}
    \int\limits_{|r|>1} |r| \sup\limits_{t\in\R} \int\limits_{0}^{1/|r|} d_{f(t,\cdot)}\left(\alpha\right)\,d\alpha \,\nu(dr)< \infty\label{Lassump}.
\end{align}
Then $X_t:=\int_{\bR^d}f(t,s)dL(s)$ is almost periodic in probability if for every $\varepsilon>0$ there exist $L_{\varepsilon}$ and $\tau\in[a,a+L_{\varepsilon}]$ for all $a\in\R$ such that
\begin{align}
    \sup\limits_{t\in\R}\|f(t,\cdot)-f(t+\tau,\cdot)\|_{L^1}+\sup\limits_{t\in\R}\|f(t,\cdot)-f(t+\tau,\cdot)\|_{L^2} <\varepsilon.\label{apinLp}
\end{align}
\end{theorem}

\begin{proof}
From [\ref{BergerMohamed}, Proposition 3] it follows that $(X_t)_{t\in\R}$ is in the domain of $L$ and that $(X_t)_{t\in\R}$ is stochastically continuous.\,Using the inequality in Lemma \ref{d1Lemma} (ii) gives a uniform estimate by taking the supremum in $t\in \R$ and choosing $R>1$ large enough such that $\int_{|r|>R} |r| \sup_{t\in\R} \int_{0}^{1/|r|} d_{f(t,\cdot)}\left(\alpha\right)\,d\alpha \,\nu(dr)$ becomes sufficiently small.\,This implies that $(X_t)_{t\in\R}$ is almost periodic in probability.
\end{proof}

In the next example we will construct functions satisfying the conditions of Theorem \ref{a1}. For $a,b\in \R$, $a\vee b$ denotes the maximum of $a$ and $b$.  

\begin{example}
Let $L$ be a L\'evy basis on $\R$ with characteristic triplet $(a,\gamma,\nu)$, $g\in L^1(\R,\R)\cap L^2(\R,\R)$ be a function satisfying
\begin{align*}
    \int\limits_{|r|>1} |r|\int\limits_0^{\frac{1}{|r|}}d_g\left(\alpha\right)d\alpha\nu(dr)<\infty
\end{align*}
and let $u:\bR\to\bR$ be an almost periodic function.\,In the examples below the stochastic continuity of $(X_t)_{t\in\R}$ follows directly from the continuity of the almost periodic function $u$ (see [\ref{BergerMohamed}, Proposition 3 ii)]).  
\begin{itemize}
\item[(i)] Assume that $f(t,x)=u(t)g(x)$. We observe that
$\{x\in\bR: |u(t)g(x)|>\alpha\}\subset\{x\in\bR:|g(x)|>\alpha/(1\vee\|u\|_\infty)\}$ and obtain
\begin{align*}
    \int\limits_{0}^{1/|r|}d_{u(t)g}(\alpha)d\alpha\le\int\limits_{0}^{1/|r|}d_{g}(\alpha/(1\vee\|u\|_\infty))d\alpha\le (1\vee\|u\|_\infty) \int\limits_0^{1/|r|}d_{g}(\alpha)d\alpha,
\end{align*}
implying that condition \eqref{Lassump} is satisfied. Moreover, it is clear that $$\|f(t,\cdot)-f(t+\tau,\cdot)\|_{L^p}=|u(t)-u(t+\tau)|\cdot\|g\|_{L^p}$$ for $p=1,2$, so that almost periodicity of $u$ implies (\ref{apinLp}).
\item[(ii)] Assume that $g$ is differentiable satisfying $g'\in L^1(\bR,\R)\cap L^2(\bR,\R)$, then $X_t:=~\int_{\bR}~g(u(t)+~x)dL(x)$ is almost periodic in probability, as $d_{g(u(t)+\cdot)}=d_g$, hence condition \eqref{Lassump} is satisfied and by Jensen's inequality
\begin{align*}
    \int\limits_{\bR} |g(u(t)+x)-g(u(t+\tau)+x)|^pdx \le& \int\limits_{\R} \left( \int\limits_{u(t)}^{u(t+\tau)} |g'(z+x)|dz\right)^pdx  \\\le& |u(t)-u(t+\tau)|^{p-1}\int\limits_{\bR}\left| \int\limits_{u(t)}^{u(t+\tau)} |g'(z+x)|^pdz\right|dx\\
    =&\|g'\|_{L^p}^p|u(t)-u(t+\tau)|^p
\end{align*}
for $p=1,\,2$.
\item[(iii)] Assume that $g$ is differentiable and for $\omega(x):=xg'(x)$ assume that
\begin{align*}
    \|\omega\|_{L^1}+\|\omega\|_{L^2}<\infty.
\end{align*}
Furthermore, let $\varepsilon>0$ and $u:\R\to\R$ be an almost periodic function such that 
\begin{align*}
    |u(t)|>\varepsilon,\quad \textrm{ for all } t\in\R.
\end{align*}
We show that $(X_t)_{t\in\R}$ defined by $X_t=\int_{\R} g(u(t)x)dL(x)$ is almost periodic in probability. To see this, observe
\begin{align*}
    \{x\in\bR: g(u(t) x)>\alpha\}=\left\{\frac{1}{u(t)}x:g(x)>\alpha\right\}=\frac{1}{u(t)}\{x:g(x)>\alpha\}
\end{align*}
and hence
\begin{align*}
    d_{g(u(t)\cdot)}=\frac{1}{u(t)} d_g\le \frac{1}{\varepsilon}d_g.
\end{align*}
Therefore, we obtain that the set $(g(u(t)\cdot))_{t\in\bR}$ satisfies condition \eqref{Lassump}.\,Assume w.l.o.g. that $u(t+\tau)>u(t)$ and calculate for $p=1,\,2$
\begin{align*}
    &\int\limits_{\R} |g(u(t)x) - g(u(t+\tau)x)|^p \, dx \\
    &=  \int\limits_{\R} \left|\, \int\limits_{[u(t),u(t+\tau)]} x g'(zx) \,dz \right|^p\,dx  \\&\le|u(t)-u(t+\tau)|^{p-1} \int\limits_{[u(t),u(t+\tau)]} \int\limits_{\R} \left|x g'(zx)\right|^p \,dx\,dz\\
    &\le |u(t)-u(t+\tau)|^{p-1}\int\limits_{[u(t),u(t+\tau)]}\frac{1}{z^{p+1}} \int\limits_{\R} \left|y g'(y)\right|^p \,dy\,dz\\
    &=|u(t)-u(t+\tau)|^{p-1}\left|\frac{1}{u(t)^p}-\frac{1}{u(t+\tau)^p}\right|\|\omega\|_{L^p(\mathbb{R})}^p\\
    &\le\frac{\|\omega\|_{L^p(\mathbb{R})}^p}{\varepsilon^{2p}}|u(t)-u(t+\tau)|^{p-1}\left|u(t+\tau)^p-u(t)^p\right|\\
    &\le 2p\frac{\|\omega\|_{L^p(\mathbb{R})}^p\|u\|_\infty^{p-1}}{\varepsilon^{2p}}|u(t)-u(t+\tau)|^{p}
    \end{align*}
and obtain that $X_t=\int_{\R} g(u(t) x) \,dL(x)$ is almost periodic in probability.
\item[(iv)] In this example, we will see that we can omit the assumption on $g$ being differentiable in (ii) and (iii) above. Given $g\in L^1(\R,\R)\cap L^2(\R,\R)$ by the denseness of $C^\infty_c(\bR)$ there exists an approximating sequence of infinitely differentiable functions with compact support $(g_n)_{n\in\N}\subset C_c^{\infty}(\R)$ such that $\|g-g_n\|_{L^1} <\varepsilon/2$ and $\|g-g_n\|_{L^2} <\varepsilon/2$ for a fixed $n$ large enough.\,Let $h:\bR\times\bR\to \bR$ be measurable. For $p=1,2$, we have 
\begin{align*}
    &\|g(h(t,\cdot))-g(h(t+\tau,\cdot))\|_{L^p}\\
    &\le  \,2\sup_{t\in \bR} \,\|g(h(t,\cdot))-g_n(h(t,\cdot))\|_{L^p}+  \|g_n(h(t+\tau,\cdot))-g_n(h(t,\cdot))\|_{L^p} 
\end{align*}
and observe for the examples (ii) and (iii) above, i.e. $h(t,x)=u(t)+x$ and $h(t,x)=u(t)x$ that the second term can be estimated with the same calculations as above, whereas now it is sufficient to assume $g\in L^1(\R,\R)\cap L^2(\R,\R)$ to obtain that $X_t=\int\limits_{\R} g(h(t, x)) \,dL(x)$ is almost periodic in probability.\,This holds since
\begin{align*}
    \sup_{t\in\bR}\|g(u(t)+\cdot)\|_{L^p}=\|g\|_{L^p}
\end{align*}
and
\begin{align*}
    \sup_{t\in\bR}\|g(u(t)\cdot)\|_{L^p}=\sup_{t\in\bR}\frac{1}{|u(t)|^{1/p}} \|g\|_{L^p}\le \frac{1}{\varepsilon^{1/p}}\|g\|_{L^p}.
\end{align*} 
Further examples such as $g(h(t,x))=\mathds{1}_{[0,a(t)]}(x)$, where $a:\R\to (0,\infty)$ is almost periodic with $\inf\limits_{t\in\R} a(t)>0$, are now also covered. To see this, notice that $$g(h(t,x))= \mathds{1}_{[0,a(t)]}(x) = \mathds{1}_{[0,1]}(u(t)x),$$ where $u(t):=\frac{1}{a(t)}$ clearly is an almost periodic function.
\end{itemize}
\end{example}

\section{almost periodic stationary ornstein-uhlenbeck-type processes}
In this section we obtain conditions for a L\'evy-driven Ornstein-Uhlenbeck process stochastic differential equation to have an almost periodic stationary solution.\,A stochastic process $(L_t)_{t\in\R}$ with values in $\R$ is called a two-sided L\'evy process if $L_0=0$ a.s., has independent and stationary increments and almost surely c\`adl\`ag paths (i.e. right-continuous paths with finite left limits). A two sided L\'evy process $L$  with characteristic triplet $(a,\gamma,\nu)$ can be constructed by two independent L\'evy processes $L',\,L''$ with same characteristic triplets and $L_t:=L'_{t\vee 0}-L''_{(-t\vee 0)-}$ for $t\in\mathbb{R}$. It naturally induces a L\'evy basis $(L(A))_{A\in \mathcal{B}_b(\R)}$ on $\R$ via $L(A)=\int_{\R} \mathds{1}_A(s)\,dL_s$ for $A\in \mathcal{B}_b(\R)$.\,An Ornstein-Uhlenbeck process $X^{(\mu)}=(X^{(\mu)}_t)_{t\in\bR}$ driven by a L\'evy process $L$ is a solution of the stochastic differential equation $dX_t=\mu X_t dt+dL_t$, where $\mu\in \bR$. As an extension of this model Alkadour studied in [\ref{Alkadour}] a periodic version of such a process. He looked at a stochastic differential equation of the form $dX_t=\mu(t) X_t dt+dL_t$, with a periodic function $\mu $. Let us consider an Ornstein-Uhlenbeck-type process with almost periodic coefficient $\mu:\R\to\R$, which solves the stochastic differential equation
\begin{align}
    dX_t&=\mu(t)X_tdt+dL_t,\label{eq1}\\
    X_0&=X(0),\nonumber
\end{align}
where $L$ is a two-sided real-valued L\'evy process and $X(0)$ some initial random variable. By (\ref{eq1}) we mean that $X_{t_2}-X_{t_1} = \int_{t_1}^{t_2} \mu(s)X_s\,ds + L_{t_2} - L_{t_1}$ is satisfied whenever $t_1<t_2$. The solution of (\ref{eq1}) is given in the next result:



\begin{proposition}\label{PropositionOU}
Let $X(0)$ be given and let $\mu:\R\to\R$ be an almost periodic function. Then the unique solution of (\ref{eq1}) is given by
\begin{align}
X_t= e^{Z_t}\left(X(0)+\int\limits_{0}^t e^{-Z_s}  dL(s)\right),\quad t\in\R,\label{solOUP}
\end{align}
where $Z_t:= \int_0^t \mu(s)ds$ with the interpretation $\int_0^t = -\int_t^0$ for $t<0$.
\end{proposition}

\begin{proof}
It follows from [\ref{Jaschke}, Theorem 1] that a solution to (\ref{eq1}) necessarily satisfies
\begin{align}
    X_{t_2} = e^{Z_{t_2} - Z_{t_1}}\left( X_{t_1} + \int\limits_{t_1}^{t_2} e^{-(Z_s-Z_{t_1})} dL(s) \right)\label{OUN}
\end{align}
for all $t_1\le t_2$. Choosing $(t_1,t_2)=(0,t)$ when $t\ge 0$ and $(t_1,t_2)=(t,0)$ when $t\le 0$ gives uniqueness.\,On the other hand, defining $X_t$ by (\ref{solOUP}), it is easily seen that $X$ satisfies (\ref{OUN}), and again by [\ref{Jaschke}, Theorem 1] we see that $X$ is a solution of (\ref{eq1}).
\end{proof}

The proof of Proposition \ref{PropositionOU} actually shows that it is enough to assume that the function $\mu:\R\to\R$ is c\`adl\`ag. Now let $\mu$ be an almost periodic function. The question now is if there exists an almost periodic stationary version of this process and if this solution is unique, i.e. if $X(0)$ can be chosen such that $(X_t)_{t\in \R}$ becomes almost periodic stationary, and if this choice of $X(0)$ is unique.\,The related question in the case when $\mu<0$ is constant (leading to the classical Ornstein-Uhlenbeck process and stationary solutions) is well-studied, and it is well-known that then stationary solutions exist if and only if $\int_{|s|>1} \log(|s|)\,\nu(ds)<\infty$, see e.g. [\ref{Jacod2}], [\ref{Sato2}], [\ref{Sato3}] or [\ref{Wolfe}]. Similarly, Alkadour studied in [\ref{Alkadour}] periodic stationary solutions when $t\mapsto \mu(t)$ is periodic. Returning to the question of almost periodic stationary solutions when $\mu$ is almost periodic, the usual candidate is a process of the form
\begin{align}\label{eq2}
X_t:= \int\limits_{\R} f(t,s) dL(s), 
\end{align}
where $f:\R\times\R\to \R$ is given by 
\begin{align}
    f(t,s):= \exp\left(\int\limits_{s}^t \mu(u) \,du \right)\mathds{1}_{(-\infty,t]}(s).\label{kernelOU}
\end{align}
With $Z$ of Proposition \ref{PropositionOU}, this is written as $X_t = \int_{-\infty}^t e^{Z_t-Z_s} dL(s)$. 
Let us show that $X$ is well-defined and almost periodic stationary under some additional conditions on $\mu$ and $L$. Since $\mu$ is almost periodic, the limit $\lim_{T\to\infty} \frac{1}{2T} \int_{-T}^T \mu(s)\,ds$ exists in $\R$ (e.g. [\ref{Levitan}, Section 2.3].

\begin{theorem}\label{OrnTheorem1}
Let $\mu:\bR\to\bR$ be an almost periodic function such that $C:=-\lim_{T\to \infty}\frac{1}{2T}\int_{-T}^{T}\mu(s)ds>0$ and let $f$ be defined as in (\ref{kernelOU}). If the L\'evy basis $L$ with characteristic triplet $(a,\gamma,\nu)$ satisfies
\begin{align}
     \int\limits_{|s|>1} \log(|s|) \,\nu(ds)<\infty,\label{logmomentOU}
\end{align}
then $f(t,\cdot)\in D(L)$, and the process $X=(X_t)_{t\in\R}$ defined by (\ref{eq2}) is the unique almost periodic stationary solution of (\ref{eq1}).  
\end{theorem}

\begin{proof}
As $\mu$ is almost periodic, there exists a $T_0>0$ such that 
\begin{align}
\frac{1}{T}\int\limits_{x}^{T+x} \mu(s)ds<-\frac{C}{2}\textrm{ for all }T\ge T_0\textrm{ and } x\in\R,\label{MeanMu}
\end{align}
see [\ref{Levitan}, Section 2.3]. Now let $s<t$ and $s\in (kT_0,(k+1)T_0]$ and $t\in (lT_0,(l+1)T_0]$, where $l,k\in \bZ$. We have
\begin{align}
    \int_s^t\mu(u)du=&\sum\limits_{m=k}^l \int\limits_{mT_0}^{(m+1)T_0}\mu(u)du - \int_{kT_0}^s \mu(u)du- \int_t^{(l+1)T_0}\mu(u)du\nonumber\\
    \le&-(l+1-k)T_0 C/2 +2 T_0 \sup_{x\in\bR}|\mu(x)|\nonumber\\
    \le&-(t-s)C/2+2 T_0 \sup_{x\in\bR}|\mu(x)|=:\frac{1}{2}C(s-t)+C'\label{estimateOU}.
\end{align}
It holds for $\alpha\in (0,1)$ 
\begin{align*}
    d_{f(t,\cdot)}(\alpha) &= \lambda^1\left( \{x\in (-\infty,t]: |f(t,x)|>\alpha\}  \right)\\
    &= \lambda^1\left(\bigg\{ x\in (-\infty,t]: \int\limits_x^t \mu(u) \,du > \log(\alpha)\bigg\}  \right)\\
    &\le  \lambda^1\left(\bigg\{ x\in (-\infty,t]: C(x-t)+2C' > 2\log(\alpha)\bigg\}  \right)\\
    &=  \frac{2\,C'}{C} - \frac{2}{C}\, \log(\alpha),
\end{align*}
and therefore
\begin{align*}
    &\int\limits_{|s|>1} |s|\sup\limits_{t\in\R}\int\limits_{0}^{1/|s|} d_{f(t,\cdot)}(\alpha) \,d\alpha \,\nu(ds) \\
    &\le \frac{2}{C}\,(1+C')\,\nu\left(\R\setminus [-1,1]\right) + \frac{2}{C}   \int\limits_{|s|>1} \log(|s|)\,\nu(ds) <\infty.
\end{align*}
Hence, $X_t$ given by (\ref{eq2}) is indeed well-defined, see [\ref{BergerMohamed}, Proposition 3i)]. 
Let us now show that $X$ defined by (\ref{eq2}) is almost periodic stationary.\,As $\mu$ is almost periodic, for every $\varepsilon>0$ there exists $L_{\varepsilon}$ and $\tau=\tau(a,\varepsilon)\in [a,a+L_{\varepsilon}]$ for all $a\in \R$ such that
\begin{align*}
    |\mu(u)- \mu(u+\tau)|<\varepsilon \textrm{ for all } u\in\R.
\end{align*}
For $|t-s|\le \frac{1}{\varepsilon}$ and $\tau=\tau(a,\varepsilon)$ we get
\begin{align*}
&|f(t+\tau,s+\tau)-f(t,s)|\\
&=\mathds{1}_{(-\infty,t]}(s)\left|\exp\left(\int\limits_{s}^t \mu(u)\,du \right)-\exp\left(\,\,\int\limits_{s+\tau}^{t+\tau} \mu(u)\,du \right)\right|\\
&= \mathds{1}_{(-\infty,t]}(s)\exp\left(\int\limits_{s}^t \mu(u)\,du \right)\left| 1-\exp\left(\int\limits_{s}^t (\mu(u+\tau)-\mu(u) )\,du\right) \right|\\
&\le \mathds{1}_{(-\infty,t]}(s)\exp\left(\int\limits_{s}^t  \mu(u)\,du \right)2\varepsilon (t-s),
\end{align*}
where we used that $|1-e^x|\le 2 |x|$ for all $|x|\le 1$.\,For general $s,t\in \R$ we obtain by the triangular inequality and (\ref{estimateOU})
\begin{align*}
    |f(t+\tau,s+\tau)-f(t,s)|\le 2\,\mathds{1}_{(-\infty,t]}(s) e^{C/2(s-t)+C'}.
\end{align*}
Let $\varepsilon\in (0,1)$. We observe for $p\ge 1$
\begin{align*}
    &\int_{\bR}|f(t+\tau,s+\tau)-f(t,s)|^pds\\
    = &\int_{(-\infty,t-\frac{1}{\varepsilon}]}|f(t+\tau,s+\tau)-f(t,s)|^pds+ \int_{(t-\frac{1}{\varepsilon},t]}|f(t+\tau,s+\tau)-f(t,s)|^pds\\
    \le &2^p\,e^{pC'} \int_{(-\infty,t-\frac{1}{\varepsilon}]} e^{\frac{pC}{2}(s-t)} ds + 2^p \int_{(t-\frac{1}{\varepsilon},t]}\exp\left(p\int\limits_{s}^t  \mu(u)\,du \right)\varepsilon^p(t-s)^p ds\\
    \le & 2^p\,e^{pC'}\int_{(-\infty,-\frac{1}{\varepsilon}]} e^{\frac{pC}{2}s} ds+2^p\varepsilon^p\,e^{pC'} \int_{(-\frac{1}{\varepsilon},0]} e^{p\frac{C}{2}s}  |s|^p ds \\
    \le & \frac{2^{p+1}\,e^{pC'}}{pC}  e^{-\frac{pC}{2\varepsilon}}+2^p\varepsilon^p \,e^{pC'} \int_{(-\infty,0]} e^{p\frac{C}{2} s} |s|^p ds.
\end{align*}
Therefore, the almost periodic stationarity of $X$ follows from Corollary \ref{a10} with $s_1(\tau)=s_2(\tau)=\tau$ (the continuity of $t\mapsto f(t,\cdot)$ in $L^1(\R,\R)\cap L^2(\R,\R)$ is easily seen). To see the uniqueness of the almost periodic stationary solution, let $Y=(Y_t)_{t\in\R}$ be an almost periodic stationary solution. By Proposition~\ref{apsProp} b), $(\mathcal{L}_{Y_t})_{t\in\R}$ is uniformly tight and with [\ref{Kallenberg}, Lemma 4.9] and (\ref{MeanMu}) we get that $e^{-Z_t} Y_t \to 0$ in probability as $t\to -\infty$, where $Z$ is defined as in Proposition \ref{PropositionOU}. Furthermore, it holds $\int^t e^{-Z_s} \, dL_s \to \int_{-\infty}^0 e^{-Z_s}\, dL_s$ in probability as $t\to -\infty$. Hence, by (\ref{solOUP}), $Y_0$ is the probability limit of $e^{-Z_t} Y_t - \int_{0}^{t} e^{-Z_s} \,dL_s$ as $t\to\infty $, which is equal to $\int_{-\infty}^0 e^{-Z_s}\,dL(s)$.\,This shows that $Y_0$, and hence $Y$, are unique.
\end{proof}

\begin{remark}
In Theorem \ref{OrnTheorem1} we saw that if the (asymptotic) mean $\lim\limits_{T\to\infty}\frac{1}{2T}\int_{-T}^T \mu(s)ds$ is negative, it is easy to construct a unique almost periodic stationary solution. In the case that the mean is positive, one can also find by the same methods a unique almost periodic stationary solution, which is then given by $X_t = -\int_{t}^\infty e^{Z_t-Z_s} dL(s)$ under the sufficient and necessary condition (\ref{logmomentOU}). If the mean is $0$, we do not know if this implies that there exists no stationary solution. If $\mu$ is periodic and the mean of $\mu$ is $0$, there exists no stationary periodic solution, see [\ref{Alkadour}].
\end{remark}

\section{Central limit theorem for almost periodic stationary processes}
In this section we prove a central limit theorem for $m$-dependent and $L^2$-uniformly integrable almost periodic stationary processes.\,We start with the following lemma (recall that $\lambda^d$ denotes the $d$-dimensional Lebesgue measure).

\begin{lemma}\label{uniform}
Let $(X_t)_{t\in\R^d}$ be an $L^2$-uniformly integrable and jointly measurable stochastic process. Denote $Y_k := \int_{A_k} X_t\, dt$, where $(A_k)_{k\in\Z}$ is a sequence of Borel measurable sets with $C:=\sup_{k\in\Z}\lambda^d( A_k)<\infty $. Then $(Y_k)_{k\in\Z}$ is $L^2$-uniformly integrable.
\end{lemma}
The joint measurability of $X$ implies that the paths $t\mapsto X_t(\omega)$ are Borel-measurable and the integrals $\int_{A_k} X_t\,dt$ exist almost surely pointwise, since
\begin{align*}
    \mathbb{E}\left( \int\limits_{A_k} |X_t|\,dt \right) = \int\limits_{A_k} \mathbb{E}(|X_t|)\,dt  \le C\, \sup\limits_{s\in\R} \mathbb{E}(|X_s|) <\infty,
\end{align*}
by Tonelli's theorem.~Tonelli's theorem then implies that the $Y_k$, $k\in\Z$, are random variables.
\begin{proof}
First define $Z_k:=\int_{A_k} X_t^2\, dt$ and observe that by Jensen's inequality for $\varepsilon>0$
\begin{align*}
    Y_k^2 \mathds{1}_{Y_k^2>\varepsilon} &\le  Z_k \mathds{1}_{|Z_k|>\frac{\varepsilon}{C}} \lambda^d(A_k) \\
    &= \left( \int\limits_{A_k} X_t^2  \mathds{1}_{ X_t^2 >\frac{\varepsilon}{2C^2}}\, dt +  \int\limits_{A_k} X_t^2 \mathds{1}_{ X_t^2 \le\frac{\varepsilon}{2C^2}}\, dt \right)\mathds{1}_{|Z_k|>\frac{\varepsilon}{C}} \lambda^d(A_k)\\
    &\le \left( \int\limits_{A_k} X_t^2 \mathds{1}_{ X_t^2 >\frac{\varepsilon}{2C^2}}\, dt +  \int\limits_{A_k} X_t^2 \mathds{1}_{ X_t^2 >\frac{\varepsilon}{2C^2}}\, dt \right)\mathds{1}_{|Z_k|>\frac{\varepsilon}{C}} \lambda^d(A_k)\\
    &\le 2 \lambda^d(A_k)\int\limits_{A_k} X_t^2 \mathds{1}_{X_t^2 >\frac{\varepsilon}{2C^2}}\, dt .
\end{align*}
It follows
\begin{align*}
    \mathbb{E} (Y_k^2 \mathds{1}_{Y_k^2>\varepsilon}) &\le 2 \lambda^d(A_k) \int\limits_{A_k} \mathbb{E}X_t^2 \mathds{1}_{X_t^2>\frac{\varepsilon}{2C^2}}\,dt\\
    &\le 2  \lambda^d(A_k) C \sup\limits_{t\in\R^d} \mathbb{E} (X_t^2 \mathds{1}_{X_t^2>\frac{\varepsilon}{2C^2}})\to 0
\end{align*}
for $\varepsilon\to\infty$, since $(X_t^2)_{t\in\R^d}$ is uniformly integrable.
\end{proof}

For $m>0$ a stochastic process $(X_t)_{t\in\R}$ is called $m$-dependent, if $(X_t)_{t\le u}$ and $(X_t)_{t>u+m}$ are independent for all $u\in\R$. In the following, we prove our central limit theorem for $m$-dependent almost periodic stationary processes.
\begin{theorem}\label{mcentral}
Let $m>0$ and $(X_t)_{t\in\R}$ be an $m$-dependent, jointly measurable, mean zero almost periodic stationary process which is $L^2$-uniformly integrable. Then
\begin{align*}
    g(s):= \lim\limits_{T\to\infty } \frac{1}{2T} \int\limits_{-T}^T \mathbb{E}X_{s+t}X_t\,dt
\end{align*}
exists for each $s\in\R$ as a limit in $\R$, the function $g:\R\to\R$ is measurable and bounded, $\int_{-m}^m g(s)\,ds \ge 0$ and 
\begin{align}
    \frac{1}{\sqrt{2T}} \int\limits_{-T}^T X_t\,dt  \overset{d}{\to} N\left( 0, \int\limits_{-m}^m g(s)\,ds \right)\quad \textrm{ as } T\to\infty.\label{convCLT}
\end{align}
\end{theorem}

\begin{proof}
Let the process $X$ be defined on the probability space $(\Omega, \mathcal{F}, P)$.\,By Proposition \ref{apsProp} d), $X$ is almost periodically correlated, so that $t\mapsto \mathbb{E}X_{s+t}X_t$ is an almost periodic function for each $s\in\R$. By the mean-value property of almost periodic functions (e.g. [\ref{Levitan}, Section 2.3]), the limit $g(s):= \lim_{T\to\infty} \frac{1}{2T} \int_{-T}^T \mathbb{E}X_{s+t}X_t\,dt$ exists in $\R$ for each $s\in\R$. That $g$ is bounded follows from the $L^2$-uniform integrability of $X$ and the Cauchy-Schwarz inequality.\,To see that $g$ is measurable, observe that the functions $\R^2\times \Omega\to \R$, $(s,t,\omega)\mapsto X_s(\omega)$ and $\R^2\times \Omega\to\R$, $(s,t,\omega)\mapsto X_t(\omega)$ are $\mathcal{B}(\R^2)\otimes \mathcal{F}$-measurable by joint measurability of $X$, hence so is $(s,t,\omega)\mapsto X_t(\omega)X_s(\omega)$. Fubini's theorem implies measurability of $\R^2 \ni (s,t)\mapsto \mathbb{E}X_tX_s$ (as well as the almost sure existence of pathwise integrals like $\int_{-T}^T \int_{-T}^T X_t(\omega) X_s(\omega) \,dt\,ds$), hence $\R\ni s\mapsto \frac{1}{2T} \int_{-T}^T \mathbb{E}X_{s+t}X_t\,dt$ is measurable for each $T>0$, hence so is $g$ as a limit of measurable functions. For $T>0$ denote 
\begin{align*}
    V_T:= \left( \mathbb{E}\left( \frac{1}{\sqrt{2T}} \int\limits_{-T}^T X_t\,dt \right)^2 \right)^{1/2} \ge 0.
\end{align*}
As $(X_t)_{t\in\R}$ is $m$-dependent, we obtain for $T>m$
\begin{align*}
    V_T^2 = \frac{1}{2T} \int\limits_{-T}^T \int\limits_{-T}^T \mathbb{E} X_tX_s\,dt\,ds = \frac{1}{2T} \int\limits_{-m}^m \int\limits_{-T}^T \mathbb{E}X_tX_{t+s} \mathds{1}_{|t+s|\le T} \,dt\,ds.
\end{align*}
Since 
\begin{align*}
    \lim\limits_{T\to\infty} \frac{1}{2T} \int\limits_{-m}^m \int\limits_{-T}^T \mathbb{E}(X_tX_{t+s})(1-\mathds{1}_{|t+s|\le T})\,dt\,ds =0
\end{align*}
an application of Lebesgue's dominated convergence theorem shows
\begin{align*}
    \lim\limits_{T\to\infty} V_T^2=\lim\limits_{T\to\infty} \frac{1}{2T} \int\limits_{-m}^m \int\limits_{-T}^T \mathbb{E}X_{s+t}X_t\,dt\,ds = \int\limits_{-m}^m g(s)\,ds=:V^2_{\infty},
\end{align*}
so that $V_{\infty}^2= \int_{-m}^m g(s)\,ds \ge 0$.\\
To show (\ref{convCLT}), we distinguish the cases $V_{\infty}^2 = 0$ and $V_{\infty}^2>0$. If $V_{\infty}^2=0$, then the above calculations show that $\frac{1}{\sqrt{2T}} \int_{-T}^T X_t\,dt$ converges in $L^2$ to $0$, hence in distribution. Thus, we may assume $V_{\infty}^2>0$. Then (\ref{convCLT}) is equivalent to
\begin{align}
    \frac{1}{V_T} \frac{1}{ \sqrt{2T} } \int\limits_{-T}^T X_t\,dt \overset{d}{\to} N( 0,1),\quad T\to\infty.\label{convCLT2} 
\end{align}
Since $\frac{1}{V_T} \frac{1}{ \sqrt{2T} } \left(\int_{\lfloor T\rfloor}^T X_t\,dt + \int_{-T}^{-\lfloor T\rfloor} X_t\,dt \right)$
converges in probability to $0$ as $T\to\infty$ by uniform integrability (here, $\lfloor x\rfloor$ denotes the integer part of $x\in\R$), it is clearly enough to prove (\ref{convCLT2}) where $T$ is restricted to the natural numbers. For that, we will apply a result of Heinrich \cite[Theorem 2, p.~135]{Heinrich}. Denote (for $n\in\N$ large enough such that $V_n>0$)
\begin{align*}
    W_n &:= \{ -n,-n+1,\dotso,n-1 \}\subset \Z,\\
    Y_k &:= \int\limits_{k}^{k+1} X_t\,dt,\,k\in\Z,\\
    U_{n,k}&:=\frac{1}{V_n} \frac{1}{ \sqrt{2n} }Y_k,\, k\in W_n \quad \textrm{ and }\\
    S_n&:= \sum\limits_{k\in W_n} U_{n,k}.
\end{align*}
Then $|W_n|\to\infty$, $\mathbb{E} U_{n,k} =0$, and $(U_{n,k})_{k\in W_n}$ is $m_n$-dependent with $m_n:=\lfloor m\rfloor+2$ (independent of $n$).\,Further, $S_n = \frac{1}{V_n} \frac{1}{ \sqrt{2n} } \int_{-n}^n X_t\,dt$ satisfies $\mathbb{E}S_n^2 = 1$ by definition of $V_n$. Next, since $(Y_k)_{k\in\Z}$ is $L^2$-uniformly integrable by Lemma \ref{uniform}, we have
\begin{align*}
    \sup\limits_{n\in\N} \sum\limits_{k\in W_n} \mathbb{E}U^2_{n,k} \le \sup\limits_{n\in\N} \left( \frac{1}{V_n^2}\frac{1}{2n} 2n \right) \sup\limits_{k\in\N} \mathbb{E}Y_k^2 <\infty
\end{align*}
and
\begin{align*}
    \sum\limits_{k\in W_n} \mathbb{E}\left( U^2_{n,k}\mathds{1}_{|U_{n,k}|\ge \varepsilon} \right) &= \frac{1}{V_n^2 2n} \sum\limits_{k=-n}^{n-1} \mathbb{E}\left( Y_k^2 \mathds{1}_{|Y_k|\ge \varepsilon V_n\sqrt{2n}} \right)\\
    &\le \frac{1}{V_n^2}\sup\limits_{k\in\Z} \mathbb{E}\left( Y_k^2 \mathds{1}_{|Y_k|\ge \varepsilon V_n\sqrt{2n}} \right)\to 0 
\end{align*}
as $k\to\infty$ for every $\varepsilon>0$. Hence, all assumptions of Theorem 2 in [\ref{Heinrich}] are satisfied, yielding 
\begin{align*}
    \frac{1}{V_n} \frac{1}{\sqrt{2n}} \int\limits_{-n}^n X_t\,dt = S_n \overset{d}{\to} N( 0,1), \textrm{ as }n\to\infty,
\end{align*}
i.e. (\ref{convCLT2}).
\end{proof}

\section*{Acknowledgement:} The first author is financially supported through the DFG-NCN Beethoven Classic 3 project SCHI419/11-1 \& NCN
2018/31/G/ST1/02252.\,The two authors would like to thank Alexander Lindner for his support, helpful comments and for many interesting and fruitful discussions, improving the content and the exposition of the paper immensely. Furthermore, the authors would like to thank Ren\'e L. Schilling for his valuable comments, which improved the presentation of the paper.

\mbox{}\\
\,\\
David Berger\\
 TU Dresden, Institute of Mathematical Stochastics, Zellescher Weg 12-14, 01069 Dresden,
Germany\\
email: david.berger2@tu-dresden.de\\\\
Farid Mohamed\\
Ulm University, Institute of Mathematical Finance,  Helmholtzstra{\ss}e 18, 89081 Ulm,
Germany\\
email: farid.mohamed@uni-ulm.de
\end{document}